\PassOptionsToPackage{colorlinks=true, linkcolor=blue, bookmarks=false}{hyperref}
\documentclass[11pt]{article}
\usepackage{latexsym,amsmath,amsthm,amssymb}
\usepackage[sorted]{amsrefs}
\usepackage{comment}
\usepackage{graphicx}
\usepackage[colorlinks=true, linkcolor=blue]{hyperref}
\usepackage{geometry}
\usepackage{enumitem}
\usepackage{mathrsfs}
\usepackage{tikz}
\usepackage{mathtools}
\usepackage{float}
\usepackage{nccmath}
\usepackage{xcolor}
\usepackage{indentfirst}
\usepackage{array,multirow,makecell}

\newtheorem{theorem}{Theorem}[section]
\newtheorem{proposition}[theorem]{Proposition}

\newtheorem{conjecture}[theorem]{Conjecture}

\newtheorem{corollary}[theorem]{Corollary}
\newtheorem{question}[theorem]{Question}
\newtheorem{obs}[theorem]{Observation}

\newtheorem{definition}[theorem]{Definition}
\newtheorem{remark}[theorem]{Remark}

\def\yoa#1{{\color{blue}{({\sc Alexander says: }{\marrow\sf #1})}}}

\def\george#1{{\color{red}{({\sc George says: }{\marrow\sf #1})}}}

\newcommand{\floor}[1]{\lfloor #1 \rfloor}
\newcommand{\Floor}[1]{\left\lfloor #1 \right\rfloor}

\title{Rainbow Separating Path Systems}
\author{Alexander Clifton\thanks{Czech Technical University in Prague, Prague, Czechia. Supported by grant 23-06815M of the Grant Agency of the Czech Republic. \texttt{alexander.clifton@fit.cvut.cz}} \and George Kontogeorgiou\thanks{Center for Mathematical Modeling (CNRS IRL2807), University of Chile. Supported by ANID Basal Grant CMM FB210005 and ANID-FONDECYT Postdoctorado Grant No. 3250479. \texttt{gkontogeorgiou@dim.uchile.cl}
} \and S Taruni\thanks{Department of Mathematics and Statistics, University of Otago. \texttt{taruni.sridhar@gmail.com}} \and Ana Trujillo-Negrete\thanks{Universidad Nacional Autónoma de México, CDMX, Mexico. Supported by the Universidad Nacional Autónoma de México Postdoctoral Program (POSDOC), and by ANID Basal Grant CMM FB210005.  \texttt{ltrujillo@ciencias.unam.mx}}}

\date{\today}

\begin{document}

\maketitle
\begin{abstract}
    We introduce a colorful version of separating path systems, in which two edges can only be separated from each other by two paths of distinct colors. We calculate the minimum sizes of such systems for various standard classes of graphs and numbers of colors. With respect to this setup, we identify three possible asymptotic behaviors for a class of graphs as the number of colors goes to infinity, and we find a wide range of examples that display each of these behaviors.    
\end{abstract}
\section{Introduction}
Let $G$ be a graph and let $\mathscr{P}$ be a set of paths of $G$. We say that $\mathscr{P}$ is a \emph{strongly separating path system} if for every ordered pair of distinct edges $(e,f)\in E(G)^2$, there exists a path in $\mathscr{P}$ that contains $e$ and avoids $f$. We say that $\mathscr{P}$ is a \emph{weakly separating path system} if for every pair of edges $\{e,f\}\in{E(G)\choose 2}$ there exists a path in $\mathscr{P}$ that contains exactly one of them. The \emph{strong} (resp. \emph{weak}) \emph{separation number} of $G$, denoted $ssp(G)$ (resp.~$wsp(G)$), is the least possible size of a path system that strongly (resp.~weakly) separates $E(G)$. Since strongly separating path systems are also weakly separating, $wsp(G)\leq ssp(G)$ for every graph $G$.

The problem of computing the separation numbers of various graphs is now over a decade old. It was first stated in its current form by Gyula O. H. Katona (see \cite{nlogn}). Falgas-Ravry, Kittipassorn, Kor\'andi, Letzter and Narayanan \cite{nlogn} were the first to publish results in the weak setting, whereas Balogh, Csaba, Martin and Pluh\'ar \cite{balogh2016path} pioneered the strong setting. Both groups conjectured a linear upper bound in the number of vertices; an elegant proof of this conjecture was later given by Bonamy, Botler, Dross, Naia and Skokan \cite{19n}, establishing the bound $ssp(G)\leq 19|V(G)|$. Moreover, these numbers have been calculated exactly for trees \cite{balogh2016path, arrepol2023separating}, asymptotically for dense, almost-regular, robustly connected graphs \cite{nico}, and up to an additive constant for complete graphs \cite{me}. Some authors have recently studied variations of this topic, such as path systems that separate vertices instead of edges \cite{biniaz2023separating, lichev2024vertex} and separation systems that consist of subdivisions of a given graph instead of paths \cite{botler2024separating}. In this paper we introduce a generalization of separating path systems, namely \emph{rainbow separating path systems}.

\begin{definition}
    A multiset of paths $\mathscr{P}$ in a graph $G$ is a $k$-rainbow separating path system (or $k$-RSPS for short) if its elements can be colored with $k\in\{2,3,\dots\}$ colors so that for every pair of edges $\{e_1,e_2\}\in {E(G)\choose 2}$ there exist paths $P_1,P_2\in \mathscr{P}$ of distinct colors such that $e_1\in P_1\not\owns e_2$ and $e_2\in P_2\not\owns e_1$. The $k$\emph{-rainbow separation number} of $G$, denoted by $c_k(G)$, is the smallest cardinality of a $k$-RSPS of $G$.   
\end{definition} 

Clearly, $c_k(G)$ is non-increasing in $k$. The first thing to notice is that strongly separating path systems are in fact RSPS's in which we are allowed the use of infinitely many colors, so we can write $ssp(G)=c_{\infty}(G)$. RSPS's are therefore the result of a natural restriction imposed on strongly separating path systems; we ask, quite simply, what if we only had some specified finite number of available colors? One might expect that this would always force us to use significantly more paths in order to separate $E(G)$. 

Let us explain exactly how we quantify this. For a graph $G$ and a number $k\in\{2,3,\dots\}\cup\{\infty\}$, we define the $k$-\emph{rainbow separation ratio} \[r_k(G):=\frac{c_k(G)}{c_{\infty}(G)}.\] As we shall see (Proposition \ref{inequality}), $r_k(G)\leq 2$ for every graph $G$ and $k\in\{2,3,\dots\}\cup\{\infty\}$. Given a class (that is, an infinite family) $\Sigma$ of graphs, the \emph{$k$-rainbow separation ratio} of $\Sigma$ is defined to be \[r_k(\Sigma):=\limsup_{G\in\Sigma} r_k(G).\] 
The $k$-rainbow separation ratio measures the penalty that we incur for attempting to separate the edges of graphs of $\Sigma$ with only $k$ (rather than infinitely many) colors.

Surprisingly, we will see that, in many cases, a small number of colors $k$ suffices to (almost) attain the strong separation number, that is, to achieve $r_k(\Sigma)=1$. For example, for any tree $T$ on $n$ vertices, four colors suffice to rainbow separate $E(T)$ with at most $n-1$ paths (Theorem \ref{fourcolortrees}), which is tight for \emph{pathy} classes of trees (i.e. trees with $o(n)$ vertices of degree at least $3$) up to a sublinear error. In particular, for the class $\mathcal{SP}$ of spiders we have $r_4(\mathcal{SP})=1$. However, three colors are not enough: we show that there exist spiders on $n$ vertices that cannot be $3$-rainbow separated with fewer than $\frac{17}{16}(n-1)-3$ paths (Theorem \ref{threecolorspiders}). In contrast, the strong separation number of a spider on $n$ vertices is $n-1$, so $r_3(\mathcal{SP})\geq\frac{17}{16}$. For two colors, the difference is even more pronounced, as stars require $2\lfloor\frac{2(n-1)}{3}\rfloor$ paths for $2$-rainbow separation (Theorem \ref{thm:stars}), so $r_2(\mathcal{SP})\geq \frac{4}{3}$. 

On the other extreme, we show that there exist classes $\Sigma$ of graphs for which, no matter how large the (finite and fixed) number $k$ of colors that we are allowed, we need many more paths to rainbow separate the edges of graphs from $\Sigma$ than to strongly separate them. That is, there exists some $c>1$ for which we have a uniform bound $r_k(\Sigma)>c$ for all $k\in\{2,3,\dots\}$. Examples of such classes include the class $\mathcal{K}$ of complete graphs (Theorem \ref{complete}), for which \[c_k(K_n)\geq \frac{3}{2}(n-O_k(1))>n+9\geq c_{\infty}(K_n),\] so that $r_k(\mathcal{K})\geq\frac{3}{2}$ for all $k\in\{2,3,\dots\}$, or in general the class $\mathcal{G}_p$ of Erd\H{o}s-R\'enyi random graphs with constant probability $p$ (Theorem \ref{densetheorem}), for which \[c_k(G(n,p))\geq \frac{3}{2}p (n-O_k(1))>\left(\sqrt{3p+1}-1+o(1)\right)n=c_{\infty}(G(n,p)),\]
so that $r_k(\mathcal{G}_p)\geq\frac{3p}{2(\sqrt{3p+1}-1)}>1$.

Finally, there are classes $\Sigma$ of graphs that fall in-between; the sequence $r_k(\Sigma)$ converges to $1$ as $k\rightarrow\infty$, without ever attaining equality. An example of such a class is the class $\mathcal{BT}$ of complete binary trees (Theorem \ref{binary}). 

Let us define precisely these distinct asymptotic behaviors. We say that a class $\Sigma$ of graphs is \emph{rainbow separable} if for every $\varepsilon > 0$, there exists $k\in\mathbb{N}$ such that $r_k(\Sigma) < 1+\varepsilon$. If a class $\Sigma$ of graphs is rainbow separable, then we define its \emph{chromatic separation number} to be \[\chi_{rs}(\Sigma):=\min\{k\in\{2,3,\dots\}\cup\{\infty\}|r_k(\Sigma)=1\}.\]

For example, the class $\mathcal{K}$ of complete graphs is not rainbow separable. The class $\mathcal{SP}$ of spiders has $\chi_{rs}(\mathcal{SP})=4$, whereas the class $\mathcal{BT}$ of complete binary trees has $\chi_{rs}(\mathcal{BT})=\infty$.  

This raises some intriguing questions. What determines whether a class $\Sigma$ of graphs is rainbow separable? If it is, what determines its chromatic separation number? Can any $k\in\{2,3,\dots\}$ be the chromatic separation number of some class?  

Some questions about the asymptotic behavior of rainbow separation ratios are handled in Section \ref{sec4}. Before that, in Section \ref{sec2} and Section \ref{sec3}, we display results concerning the $k$-rainbow separation numbers of various standard classes of graphs for $k=2$, $3$ and $4$. Specifically, we exactly determine or tightly bound these numbers for paths, cycles, stars, spiders, and pathy trees. \\In this manner, we obtain many working examples, establish useful tools, and assert that the rainbow separation numbers are essentially distinct from the classic ones. 

All of our graphs are simple, and for any graph denoted using subscripts, the first subscript refers to the number of vertices. 

\subsection{Preliminaries}

Here we state some known results that are used throughout the paper. The first one concerns the exact strong separation numbers of trees. 

\begin{proposition}(Theorem 5, \cite{balogh2016path})\label{trees-strong}
    Let $T$ be a tree with $d_1$ leaves and $d_2$ vertices of degree $2$. If $T$ is a path, then $ssp(T)=d_1+d_2-1$. Otherwise, $ssp(T)=d_1+d_2$. 
\end{proposition}

In the following proposition, the lower bound is extracted from the proof of Theorem 1.6 in Falgas-Ravry et al. \cite{nlogn}. The upper bound is Theorem 3.3 from Arrepol et al. \cite{arrepol2023separating}.  

\begin{proposition}\label{trees-weak}
    Let $T$ be a tree with $d_1$ leaves and $d_2$ vertices of degree $2$. Then we have $ \lceil{\frac{2(d_1-1)+d_2}{3}\rceil}\leq wsp(T)\leq\max\{\lceil\frac{2d_1+d_2}{3}\rceil,\lceil\frac{d_1+d_2}{2}\rceil\}$.  
\end{proposition}

The lower bound of the next proposition appeared for the first time in Falgas-Ravry et al. \cite{nlogn}, whereas the upper bounds were proved by the second author and Stein \cite{me}.

\begin{proposition}\label{cliques}
    For the clique $K_n$, we have $n-1\leq wsp(K_n)\leq n+1$ and $n-1\leq ssp(K_n)\leq n+9$.
\end{proposition}

Finally, we will need the following simple inequality.

\begin{proposition}\label{inequality}
    For every graph $G$ and $k\in\{2,3\dots\}$, $$\frac{k}{k-1}wsp(G)\leq c_k(G)\leq c_2(G)\leq wsp(G)+ssp(G).$$
\end{proposition}

\begin{proof}
For the lower bound we observe that, in every $k$-RSPS of any graph, the set of all paths in any $k-1$ colors is weakly separating. If we denote by $\mathcal{F}_k$ an optimal $k$-RSPS of $G$, and by $p_i$ the number of paths of color $i$ in $\mathcal{F}_k$, then this observation yields $k$ inequalities of the form

\[c_k(G)-p_i=|\mathcal{F}_k|-p_i\geq wsp(G),\]
that we can subsequently sum by parts and divide by $k-1$ to obtain the desired result.  

For the upper bound, given a strongly and a weakly separating path system on a graph, we can color them in two distinct colors to obtain a $2$-RSPS. 
\end{proof}

\section{Two colors} \label{sec2}

\subsection{Paths, cycles and stars}\label{subsection}

\begin{theorem}\label{thm: paths}
For the path $P_n$, we have $c_2(P_3)=2$, $c_2(P_n)=n$ for $n=4,5,6$, and $c_2(P_n)=n+1$ for $n\ge 7$.
\end{theorem}

\begin{proof}

We firstly prove the upper bounds by explicit construction. For $3\leq n\leq 7$, refer to Figure~\ref{fig:paths-upper}.

\begin{figure}[h!]
    \centering
    \includegraphics[width=0.6\textwidth]{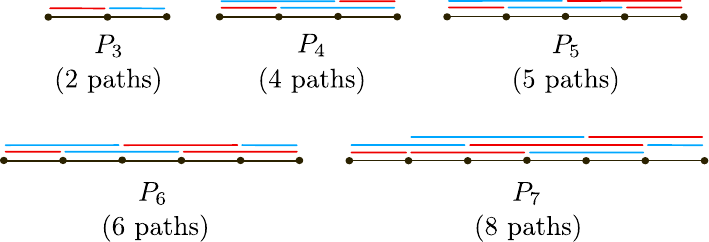}
    \caption{Optimal $2$-RSPS's for $P_n$, with $n=3,\dots,7$. }
    \label{fig:paths-upper}
\end{figure}

Consider the path $P_n=v_1\dots v_n$ for $n \geq 7$, and let $e_i=v_{i}v_{i+1}$ for $1\le i\le n-1$. For $0\le i\le n$, let $c_i$:=red if $i\equiv 0\pmod{2}$ and $c_i$:=blue if $i\equiv 1\pmod{2}$. We construct a $2$-RSPS $\mathcal{F}_n$ for $P_n$ satisfying the following properties: 
\begin{enumerate}[label=(\roman*)]
    \item \label{paths-1} Each edge belongs to at least one red path and one blue path.
    \item \label{paths-2} $e_{n-1}$ belongs to a single-edge $c_n$-path and $e_{n-2}$ belongs to a $c_n$-path that does not contain $e_{n-1}$.
    
\end{enumerate}

We proceed by induction on $n$. The path system shown in Figure~\ref{fig:paths-upper} for $P_7$ satisfies conditions~\ref{paths-1}~and~\ref{paths-2}, providing the base case. 
 
Assuming that $\mathcal{F}_n$ is a $2$-RSPS for $P_n$ satisfying~\ref{paths-1}~and~\ref{paths-2}, we construct a path system $\mathcal{F}_{n+1}$ for $P_{n+1}$ as follows:
\begin{enumerate}
    \item Extend the single-edge $c_n$-path containing $e_{n-1}$ to include $e_{n}$, and include it in $\mathcal{F}_{n+1}$. 
    \item Extend the $c_n$-path that contains $e_{n-2}$ but not $e_{n-1}$ to include $e_{n-1}$, and include it in $\mathcal{F}_{n+1}$.
    \item Include all the remaining paths from $\mathcal{F}_n$ in $\mathcal{F}_{n+1}$ without modification.
    \item Add a new single-edge $c_{n+1}$-path containing only $e_{n}$.
\end{enumerate}
See Figure~\ref{fig:paths_extension} for an example where $n$ is odd. 

\begin{figure}[h!]
    \centering
    \includegraphics[width=0.8\textwidth]{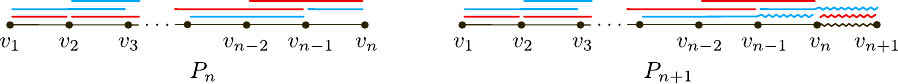}
    \caption{Inductive construction of a $2$-RSPS for $P_{n+1}$ from a $2$-RSPS for $P_n$.}
    \label{fig:paths_extension}
\end{figure}

Let us now show that $\mathcal{F}_{n+1}$ satisfies properties \ref{paths-1}~and~\ref{paths-2}. By the inductive hypothesis and the construction, each edge $e_i$ with $1\le i\le n-1$ belongs to both a red path and a blue path. \newpage Furthermore, $e_{n}$ belongs to the newly added single-edge $c_{n+1}$-path and to the extended $c_n$-path (which contains $e_{n-1}$). Moreover, $e_n$ belongs to a single-edge $c_{n+1}$-path, while $e_{n-1}$ belongs to a $c_{n+1}$-path that does not include $e_{n}$. 

We now show that $\mathcal{F}_{n+1}$ is a $2$-RSPS for $P_{n+1}$. Let $e_i$ and $e_j$ be distinct edges with $1\le i<j\leq n-1$. By the inductive hypothesis, there exist two paths in $\mathcal{F}_n$ of different colors that separate $e_i$ and $e_j$. These paths are contained in $\mathcal{F}_{n+1}$, and even if one of them is extended to include $e_n$, they still separate the edges $e_i$ and $e_j$. Moreover, $e_{n}$ is contained in a single-edge $c_{n+1}$-path, while every other edge belongs to a $c_n$-path that does not include $e_{n}$, ensuring that $e_{n}$ is separated from all other edges. 

Finally, $|\mathcal{F}_7|=8$, and we have $|\mathcal{F}_{n}|=|\mathcal{F}_{n-1}|+1$, which implies $|\mathcal{F}_n|=n+1$ for all $n\ge 7$.  

Now we prove the corresponding lower bounds. In order to separate $e_{n-2}$ from $e_{n-1}$, a $2$-RSPS must always contain a path consisting of just the edge $e_{n-1}$. Such a path does not separate pairs of edges among $e_1,\dots,e_{n-2}$, so $c_2(P_n)\ge c_2(P_{n-1})+1$ for $n\ge 3$. Thus, to establish matching lower bounds, it suffices to show that $c_2(P_3)\ge 2$, $c_2(P_4)\ge 4$, and $c_2(P_7)\ge 8$.

For $P_3$, $c_2(P_3)\geq ssp(P_3)=2$. For $P_4$, there must be a path consisting solely of $e_1$ and another consisting solely of $e_3$. As a $2$-RSPS must contain a path with $e_2$ but not $e_1$ as well as a path with $e_2$ but not $e_3$, we either have at least four paths, or the third path must consist solely of $e_2$. However, in the latter case, at least two of the single-edge paths are the same color, so the corresponding edges are not separated, meaning that $c_2(P_4)\ge 4$.

Finally, we show that $c_2(P_7)\ge 8$. For the sake of contradiction, consider a $2$-RSPS of $P_7$ consisting of at most seven paths. Among these, consider a collection $\mathcal{F}$ of $A$ red paths and $B$ blue paths of length $1$ such that no edge is contained in two of them. Since there must be single-edge paths covering $e_1$ and $e_6$, we can assume that $2\le A+B\le 6$, and we assume without loss of generality that $B\ge A$. Ignoring the $A+B$ paths in $\mathcal{F}$, the remaining paths in the $2$-RSPS form a red weakly separating path system on the $6-A$ edges not covered by the red paths in $\mathcal{F}$ as well as a blue weakly separating path system on the $6-B$ edges not covered by the blue paths in $\mathcal{F}$. 

For any set of $m>0$ edges, the size of a weakly separating path system is at least $\lceil{\log_2 m\rceil}$. Thus, for $A<6$, there are at least $A+\lceil{\log_2 (6-A)\rceil}$ red paths and when $B<6$, there are at least $B+\lceil{\log_2 (6-B)\rceil}$ blue paths. Noting that $x+\lceil{\log_2(6-x)\rceil}\ge 4$ for $x=1,2,3,4,5$, there are at least four paths of each color and hence eight total, unless $A=0$. The $2$-RSPS has at least three red paths in order to contain a red weakly separating path system, so we may assume $2\le B\le 4$. The size of the $2$-RSPS is at least $B+\lceil{\log_2 (6-B)\rceil}+3$. For $B=3,4$, this is at least $8$, so we may assume $B=2$. 

In this case, there is a blue single-edge path covering $e_1$ and a blue single-edge path covering $e_6$. The remaining blue paths form a weakly separating path system for $P_5$, so there are at least two. As there are at least three red paths, we necessarily have only two remaining blue paths and exactly three red paths in order for the $2$-RSPS to have size $7$.  Note that besides the single-edge paths covering $e_1, e_6$, we may assume that there are no other single-edge paths, as their presence would have allowed us to choose $A,B$ differently and already be done.

A blue weakly separating path system of size $2$ for $P_5$ requires one edge to be covered twice, two to be covered once (by different blue paths), and the last to not be covered at all. Up to symmetry, this can only be achieved by a path covering the first two edges and a path covering the middle two edges. That is, one remaining blue path must consist of $e_3$ and $e_4$, and the other must consist of $e_2$, $e_3$, and possibly $e_1$.

Thus, $e_5$ is not in any blue paths, so there must be a red path containing $e_5$, but not $e_4$. As all of our remaining paths have at least two edges, we must have a red path consisting of $e_5$ and $e_6$. The remaining two red paths must weakly separate the edges $e_1,e_2,e_3,e_4$. By our previous discussion for weakly separating $P_5$, one such red path must consist precisely of $e_2$ and $e_3$. At this point, there is no path containing $e_2$ but not $e_3$, and as our remaining paths have at least two edges each, the only way to do this is via a red path containing precisely $e_1$ and $e_2$. Now we have already specified seven paths, but still there is no path containing $e_4$ but not $e_3$. This yields a contradiction, so $c_2(P_7)\ge 8$.

\end{proof}

\begin{theorem}\label{thm: cycles}
	$c_2(C_n)=2\lceil{\frac{n}{2}\rceil}$.
\end{theorem}
\begin{proof}
	We firstly establish the lower bound. Let $\mathcal{F}_n$ be a $2$-RSPS for $C_n$. Consider three consecutive vertices  $u, v, w$  along the cycle. In  $\mathcal{F}_n$, there must be a path containing the edge  $uv$  but not  $vw$, and  another path of the other color containing  $vw$  but not  $uv$. This ensures that each vertex $v$ must be an endpoint of both a red path and a blue path.
	That is, the red paths collectively have at least $n$ endpoints, so there are at least $\lceil{\frac{n}{2}\rceil}$ red paths, and similarly for blue paths. Thus,  $|\mathcal{F}_n|\ge 2\lceil{\frac{n}{2}\rceil}$. 

	For the upper bound, we construct a $2$-RSPS $\mathcal{F}_n$ for $C_n$. Enumerate the edges of $C_n$ clockwise as $e_1,e_2,\dots,e_n$. Here, indices are taken modulo $n$. 
    We distinguish two cases. 
    \begin{itemize}
        \item Suppose that $n$ is even.  For each $1\le i \le \frac{n}{2}$ include in $\mathcal{F}_n$: 
	\begin{enumerate}
		\item A red path that starts at edge $e_i$ and continues clockwise through $e_{i+\frac{n}{2}-1}$.
		\item A blue path that starts at edge $e_{i+\frac{n}{2}}$ and continues clockwise to $e_{i-1}$.
	\end{enumerate} 
	Figure~\ref{fig:cycles} shows the construction of $\mathcal{F}_8$ applied to the cycle $C_8$.

       \begin{figure}[h!]
	\centering
	\includegraphics[width=0.6\textwidth]{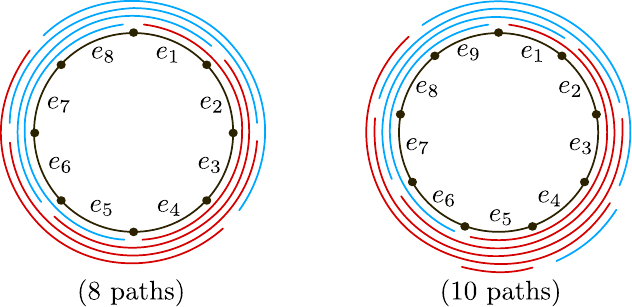}
	\caption{Optimal $2$-RSPS's for $C_8$ and $C_9$. }
	\label{fig:cycles}
\end{figure}
	
	Next, we show that $\mathcal{F}_n$ is a $2$-RSPS. Any pair of edges can be written as $e_a$ and $e_{a+b}$ for some $1 \leq b \leq \frac{n}{2}$. 
	Observe that, in $\mathcal{F}_n$,  there is a path clockwise from $e_{a-\frac{n}{2}+1}$ to $e_{a}$ and a path clockwise from $e_{a+1}$ to $e_{a+\frac{n}{2}}$. 
	Furthermore, these paths are of  different colors and separate $e_a$ from $e_{a+b}$. Thus, $\mathcal{F}_n$ is a $2$-RSPS with  \[|\mathcal{F}_n|=n=2\left\lceil\frac{n}{2}\right\rceil. \]
	\newpage
	\item Suppose that $n$ is odd. For each $1\le i \le \frac{n-1}{2}$, include in $\mathcal{F}_n$: 
	\begin{enumerate}
		\item A red path that starts at edge $e_i$ and continues clockwise through $e_{i+\frac{n-1}{2}}$.
		\item A blue path that starts at edge $e_{i+\frac{n+1}{2}}$ and continues clockwise through $e_{i-1}$.
	\end{enumerate} 
	Additionally, include in $\mathcal{F}_n$: 
	\begin{enumerate}[resume]
		\item A blue path consisting only of the edge $e_{\frac{n-1}{2}}$. 
		\item A red path consisting only of the edge $e_{\frac{n+1}{2}}$. 
	\end{enumerate} 
        See Figure~\ref{fig:cycles} for the construction of $\mathcal{F}_9$ for the cycle $C_9$.\
     
	 We now prove that $\mathcal{F}_n$ is a $2$-RSPS. 
    Consider two edges $e_a$ and $e_b$ in $C_n$ with $a<b$. We distinguish the following cases. 
	\begin{itemize}
		\item Suppose that $1\le  a \le \frac{n-1}{2}$. 
		\begin{itemize}
			\item If $b \ge a+\frac{n+1}{2}$, then in $\mathcal{F}_n$, there is a red path clockwise from $e_a$ to $e_{a+\frac{n-1}{2}}$ and a blue path clockwise from $e_{a+\frac{n+1}{2}}$ to $e_{a-1}$. The red path contains $e_a$ but not $e_b$, while the blue path contains $e_b$ but not $e_a$. Then, these paths are of different colors and separate the edges $e_a$ and $e_b$.
			\item If $a\leq\frac{n-3}{2}$ and $a<b \le a+\frac{n-1}{2}$, then in $\mathcal{F}_n$, there is a red path clockwise from $e_{a+1}$ to $e_{a+\frac{n+1}{2}}$ and a blue path clockwise from $e_{a+\frac{n+3}{2}}$ to $e_a$. The red path contains $e_b$ but not $e_a$, while the blue path contains $e_a$ but not $e_b$. These paths are of different colors and separate the edges $e_a$ and $e_b$.
            \item If $a=\frac{n-1}{2}$ and $\frac{n+3}{2}\leq b\leq n-1$, then in $\mathcal{F}_n$, there is a red path clockwise from $e_1$ to $e_{\frac{n+1}{2}}$ and a blue path clockwise from $e_{\frac{n+3}{2}}$ to $e_n$. The red path contains $e_a$ but not $e_b$, while the blue path contains $e_b$ but not $e_a$. These paths are of different colors and separate the edges $e_a$ and $e_b$.
            \item If $a=\frac{n-1}{2}$ and $b = \frac{n+1}{2}$, then there is a blue path consisting only of the edge $e_a$ and a red path consisting only of the edge $e_b$. These two single-edge paths separate $e_a$ and $e_b$.
		\end{itemize}

        \item Suppose that $\frac{n+1}{2}\leq a< n$. Then in $\mathcal{F}_n$ there is a red path clockwise from $e_{a-\frac{n-1}{2}}$ to $e_a$, and a blue path clockwise from $e_b$ to $e_{b-\frac{n+3}{2}}$. Once again, as $b-\frac{n+3}{2}\leq \frac{n-3}{2}<a$, these two paths of different colors separate $e_a$ and $e_b$.

	\end{itemize}
    We conclude that $\mathcal{F}_n$ is a $2$-RSPS with 
	\[
	|\mathcal{F}_n| = (n-1) + 2 = 2\left\lceil\frac{n}{2}\right\rceil.
	\]
    \end{itemize}
    
\end{proof}

For both paths and cycles, $c_2$ is extremely close to $ssp$. That is, for the class $\mathcal{P}$ of paths and the class $\mathcal{C}$ of cycles, we have $r_2(\mathcal{P})=r_2(\mathcal{C})=1$. However, this is not always true. By Proposition~\ref{cliques} and Proposition \ref{inequality}, $2n-2\leq c_2(K_n)\leq 2n+10$. That is, for the class $\mathcal{K}$ of complete graphs, $r_2(\mathcal{K})=2$. Another counterexample is the class of stars. By Proposition \ref{trees-strong}, a star on $n$ vertices has strong separation number $n-1$. However:
\newpage
\begin{theorem}
\label{thm:stars}
    For the star $S_n$ with $n\ge 4$, we have $c_2(S_n)=2\lfloor\frac{2(n-1)}{3}\rfloor$.
\end{theorem}
\begin{proof}
The weak separation number for $S_n$ was established by Falgas-Ravry et al.~\cite{nlogn}, who showed that $wsp(S_n)=\left\lfloor\frac{2(n-1)}{3}\right\rfloor$. Then, by Proposition~\ref{inequality}, we obtain the lower bound $c_2(S_n)\ge 2\left\lfloor\frac{2(n-1)}{3}\right\rfloor$. 

For the upper bound, we construct a $2$-RSPS $\mathcal{F}_n$ for $S_n$. Let $e_1, e_2, \ldots e_{n-1}$ denote the edges of $S_n$. The constructions of $\mathcal{F}_n$ for $n=4,5,6$ are shown in Figure~\ref{fig:stars}. 

\begin{figure}[h!]
	\centering
	\includegraphics[width=0.6\textwidth]{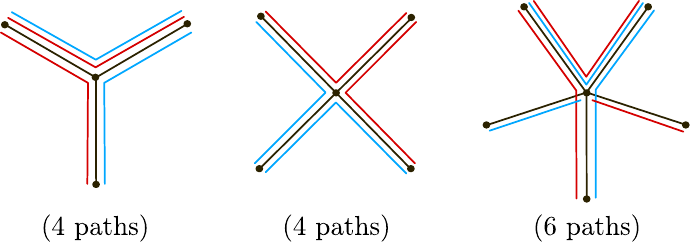}
	\caption{Optimal $2$-RSPS's for the stars $S_4$, $S_5$ and $S_6$. }
	\label{fig:stars}
\end{figure}

Now, for $n\ge 7$, we define $\mathcal{F}_n$ inductively as follows. Include in $\mathcal{F}_n$: 
\begin{enumerate}
    \item the paths from $\mathcal{F}_{n-3}$ applied to the edges $\{e_1,\dots,e_{n-4}\}$, 
    \item the four paths from $\mathcal{F}_{4}$ applied to the edges $\{e_{n-3},e_{n-2},e_{n-1}\}$. 
\end{enumerate}
As $\mathcal{F}_n$ is already defined for $n=4,5,6$ (see Figure~\ref{fig:stars}), the base case of the inductive construction is established. 

We now show that $\mathcal{F}_n$ is a $2$-RSPS. Fix two distinct edges $e_i$ and $e_j$ of $S_n$. We distinguish the following cases. 
\begin{itemize}
    \item If $i,j\in \{1,\dots,n-4\}$, then by the induction hypothesis, the paths from $\mathcal{F}_{n-3}$ applied to the edges $\{e_1,\dots,e_{n-4}\}$ and included in $\mathcal{F}_n$ separate $e_i$ and $e_j$. 
    \item If $i,j\in \{n-3,n-2,n-1\}$, then the four paths from $\mathcal{F}_{4}$ applied to the edges $\{e_{n-3},e_{n-2},e_{n-1}\}$ separate $e_i$ and $e_j$. 
    \item If $i\in \{1,\dots, n-4\}$ and $j\in \{n-3,n-2,n-1\}$, then $e_i$ belongs to a path of some color $c$ that does not include $e_j$, while $e_j$ belongs to both a red path and a blue path, neither of which contains $e_i$. Thus, we may choose the path containing $e_j$ to be of color different from $c$, ensuring that $e_i$ and $e_j$ are separated.
\end{itemize}

Finally, we determine the size of $\mathcal{F}_n$. We proceed by induction on $n$. For $n=4,5,6$, we have $c_2(S_n) = 2 \lfloor{\frac{2n}{3}}\rfloor$. Suppose now that the upper bound holds for all $k < n$. From the construction, we obtain: 
\[|\mathcal{F}_n| = |\mathcal{F}_{n-3}| + 4 \le  2 \left\lfloor{\frac{2(n-4)}{3}}\right\rfloor + 4 
= 2 \left(\left\lfloor{\frac{2(n-4)}{3}}\right\rfloor + 2 \right) = 2 \left\lfloor{\frac{2(n-1)}{3}}\right\rfloor. \]
\end{proof}

\subsection{Spiders}

For all of the classes mentioned in Subsection \ref{subsection}, $c_2$ is suspiciously close to $2\cdot wsp$, the lower bound from Proposition~\ref{inequality}. This is not true in general. For a spider of the form $S_{n,q}$, i.e. with $n$ vertices and $q\geq 3$ legs, if all of its legs are of length at least $2$, then one can construct a weakly separating path system of size $\lceil\frac{n-1}{2}\rceil$ by Proposition \ref{trees-weak}. However,

\begin{proposition}
\label{prop:short-spider}
    Let $S_{n,q}$ be a spider with $q\ge 3$ legs, all of length two, so that $n=2q+1$. Then 
    \[c_2(S_{n,q})= \begin{cases}
        n-1+\Floor{\frac{2q}{3}} & \textrm{if $q\equiv 1\pmod{3}$ or $q=3$}, \vspace{0.2cm} \\
        n-2+\Floor{\frac{2q}{3}} & \textrm{otherwise}.
    \end{cases}
    \]
\end{proposition}

\begin{proof} 
    For the upper bound, we construct an explicit $2$-RSPS $\mathcal{F}_{n,q}$ for $S_{n,q}$, following an approach similar to that used for stars in the proof of Theorem~\ref{thm:stars}.

    Let $\ell_1, \dots, \ell_q$ denote the legs of $S_{n,q}$. For $q = 3, 4, 5, 6$, the constructions of $\mathcal{F}_{n,q}$ are shown in Figure~\ref{fig:spiders_base_optimal}. 
    
    \begin{figure}[h!]
    \centering
    \includegraphics[width=0.6\textwidth]{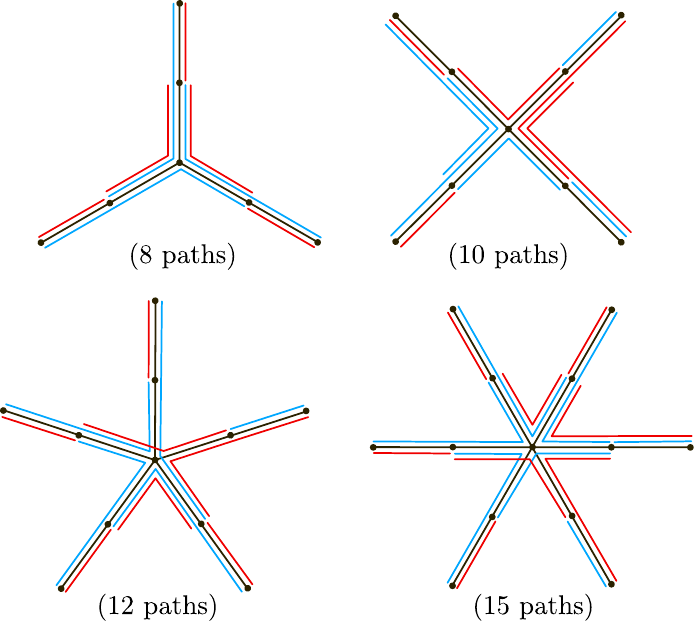}
    \caption{Optimal $2$-RSPS's for the spiders $S_{2q+1,q}$, where $q=3,4,5,6$.}
    \label{fig:spiders_base_optimal}
\end{figure}
    
    For $q \ge 7$, we define $\mathcal{F}_{n,q}$ inductively as follows. Include in $\mathcal{F}_{n,q}$:
\begin{enumerate}
    \item the paths from $\mathcal{F}_{n-6, q-3}$ applied to the legs $\{\ell_1, \dots, \ell_{q-3}\}$;
    \item the paths from $\mathcal{F}_{7,3}$ applied to the legs $\{\ell_{q-2}, \ell_{q-1}, \ell_q\}$.
\end{enumerate}

We now show that $\mathcal{F}_{n,q}$ is a $2$-RSPS. Fix two distinct edges $e$ and $e'$ of $S_{n,q}$, and let $\ell_i$ and $\ell_j$ be the legs containing $e$ and $e'$, respectively. We consider the following cases:
\begin{itemize}
    \item If $i, j \in \{1, \dots, q - 3\}$, the paths from $\mathcal{F}_{n-6, q-3}$ separate $e$ and $e'$.
    \item If $i, j \in \{q - 2, q - 1, q\}$, the paths from $\mathcal{F}_{7,3}$ separate $e$ and $e'$.
    \item If $i \in \{1, \dots, q - 3\}$ and $j \in \{q - 2, q - 1, q\}$, then $e$ belongs to a path of some color $c$ not containing $e'$, and $e'$ belongs to both a red and a blue path, neither of which contains $e$. We may thus choose a path containing $e'$ whose color is different from $c$, ensuring that $e$ and $e'$ are separated.
\end{itemize}

We now bound the size of $\mathcal{F}_{n,q}$ by induction on $q$. For $q = 3, 4, 5, 6$, the result holds by Figure~\ref{fig:spiders_base_optimal}. Assume that $q \ge 7$ and that the bound holds for all $q' < q$. From the construction we have
\[
|\mathcal{F}_{n,q}| = |\mathcal{F}_{n-6, q-3}| + |\mathcal{F}_{7,3}|.
\]
Thus,
\[
|\mathcal{F}_{n,q}| =
\begin{cases}
    n - 7 + \left\lfloor \frac{2(q - 3)}{3} \right\rfloor + 8 = n - 1 + \left\lfloor \frac{2q}{3} \right\rfloor & \text{if } q \equiv 1 \pmod{3}, \smallskip \\ 
    n - 8 + \left\lfloor \frac{2(q - 3)}{3} \right\rfloor + 8 = n - 2 + \left\lfloor \frac{2q}{3} \right\rfloor & \text{if } q \equiv 0,2 \pmod{3}.
\end{cases}
\]

To prove the lower bound, let $\mathcal{F}$ be a minimum $2$-RSPS for $S_{n,q}$. Let $v_1, \dots, v_q$ be the leaves of $S_{n,q}$, and let $e_1, \dots, e_q$ be the edges incident to them, respectively. Let $\mathcal{S}$ be the set of single-edge paths in $\mathcal{F}$ containing only $e_i$ for some $i$. Since each $e_i$ must be separated from the all the other edges in the graph, $|\mathcal{S}| \ge q$.

Let $r$ and $b$ be the number of red and blue paths in $\mathcal{S}$, respectively. Define $I_R$ and $I_B$ as the sets of indices of red and blue paths in $\mathcal{S}$, respectively. Let $T_R$ (resp. $T_B$) be the tree obtained from $S_{n,q}$ by deleting the leaves in $\{v_i : i \in I_R\}$ (resp. $\{v_j : j \in I_B\}$). Note that
\begin{enumerate}[label=$(\roman*)$]
    \item \label{item:T_R} $T_R$ has $q$ leaves and $q - r$ vertices of degree two.
    \item \label{item:T_B} $T_B$ has $q$ leaves and $q - b$ vertices of degree two.
\end{enumerate}

    Let $\mathcal{F}_R$ and $\mathcal{F}_B$ be the families obtained from $\mathcal{F}$ by:
    \begin{itemize}
        \item[$1.$] deleting the $r$ (resp. $b$) single-edge red (resp. blue) paths; and
        \item[$2.$] shortening all other red (resp. blue) paths ending in some $v_i$ for $i\in I_R$ (resp. $i\in I_B$) to the unique neighbor of $v_i$.
    \end{itemize}
    Observe that in the second step, no path is deleted or shortened to a trivial path (i.e. a path with a single vertex), 
    so every path from $\mathcal{F}$ remains in either $\mathcal{F}_R$, $\mathcal{F}_B$, or $\mathcal{S}$. This implies that the non-truncated versions of the paths in $\mathcal{F}_R$ and $\mathcal{F}_B$, together with $\mathcal{S}$, form a partition of $\mathcal{F}$. Furthermore, since $\mathcal{F}$
    was chosen as a $2$-RSPS for $S_{n,q}$ of minimum size, we have 
    \begin{equation}
    \label{eq:partition}
        c_2(S_{n,q})= |\mathcal{F}|= |\mathcal{F}_R|+|\mathcal{F}_B|+|\mathcal{S}|. 
    \end{equation}

    Note that
\begin{equation}
\label{eq:red-weak}
    \text{the red paths in $\mathcal{F}_R$ correspond to a weakly separating path system of $T_R$.}
\end{equation}
Indeed, for each pair of edges $e$ and $e'$ of $T_R$, there must exist (at least) one red path in $\mathcal{F}$ containing exactly one of these edges (because $\mathcal{F}$ is a $2$-RSPS for $S_{n,q}$), and such a path (after possibly being shortened in the second step) is included in $\mathcal{F}_R$. Similarly,
\begin{equation}
\label{eq:blue-weak}
    \text{the blue paths in $\mathcal{F}_B$ correspond to a weakly separating path system of $T_B$.}
\end{equation} 

Using property~\ref{item:T_R} with~\eqref{eq:red-weak} and applying Proposition \ref{trees-weak} to $T_R$, we get
\begin{equation}
\label{eq:weak-red}
    |\mathcal{F}_R| \ge wsp(T_R) \ge \frac{2(q - 1) + q - r}{3} = q - \frac{r + 2}{3}.
\end{equation}
Similarly, combining~\ref{item:T_B} with~\eqref{eq:blue-weak} and applying Proposition \ref{trees-weak} to $T_B$, we obtain
\begin{equation}
\label{eq:weak-blue}
    |\mathcal{F}_B| \ge wsp(T_B) \ge \frac{2(q - 1) + q - b}{3} = q - \frac{b + 2}{3}.
\end{equation}

Combining with~\eqref{eq:partition} we obtain that 
\[
    c_2(S_{n,q}) \ge  q-\frac{r+2}{3}+ q-\frac{b+2}{3}+r+b = 2q+\frac{2(r+b)-4}{3},\]
and given that $c_2(S_{n,q})$ is an integer, we conclude that
    \[c_2(S_{n,q})\ge n-1 + \left\lceil\frac{2q-4}{3}\right\rceil.
\]
We distinguish the following cases. 
    \begin{itemize}
        \item If $q\equiv 0 \pmod 3$ but $q\ne  3$, or if $q\equiv 2\pmod 3$, then $\left\lceil\frac{2q-4}{3}\right\rceil=\left\lfloor \frac{2q}{3}\right\rfloor -1$, so
        \[c_2(S_{n,q})\ge n-1 + \left\lceil\frac{2q-4}{3}\right\rceil= n-2+ \left\lfloor \frac{2q}{3}\right\rfloor. \]
        \item If $q\equiv 1 \pmod 3$, then $\left\lceil\frac{2q-4}{3}\right\rceil=\left\lfloor \frac{2q}{3}\right\rfloor$, so
        \[c_2(S_{n,q})\ge n-1 + \left\lceil\frac{2q-4}{3}\right\rceil= n-1+ \left\lfloor \frac{2q}{3}\right\rfloor. \]

        \item If $q = 3$, we aim to show that $c_2(S_{7,3}) \ge 8$. First, suppose that $r+b\ge 4$. Then, we immediately get
        \[c_2(S_{n,q})\ge 2q+\left\lceil\frac{2(r+b)-4}{3}\right\rceil\ge 8. \]
        Now assume that $r+b= 3$, which implies that all the single-edge paths in $\mathcal{S}$ are distinct. 
        By the symmetry between red and blue, we may assume that $r \in \{2, 3\}$.
        Observe that 
        \begin{itemize}
        \item  $T_R$ has $q$ leaves and $q-r$ vertices of degree two; 
        \item  $T_B$ has $q$ leaves and $q-b$ vertices of degree two. 
    \end{itemize}
    
    Consider first the case $r = 3$ and $b = 0$. Then, from~\eqref{eq:weak-red} and~\eqref{eq:weak-blue}, and using that $|\mathcal{F}_R|$ and $|\mathcal{F}_B|$ must be integers, we obtain $|\mathcal{F}_R| \ge 2$ and $|\mathcal{F}_B| \ge 3$. Therefore, by~\eqref{eq:partition}, we get $c_2(S_{7,3}) = |\mathcal{F}| \ge 8$.

    Now suppose that $r = 2$ and $b = 1$. From~\eqref{eq:weak-red}~and~\eqref{eq:weak-blue}, we have $|\mathcal{F}_R| \ge wsp(T_R)\ge 2$ and 
    $|\mathcal{F}_B| \ge wsp(T_B) \ge 2$. However, it can be easily verified that $T_B$ (which is a spider with three legs of lengths $2$, $2$, and $1$) cannot be weakly separated with only two paths. Hence, $|\mathcal{F}_B| \ge 3$, and it follows from~\eqref{eq:partition} that 
    $c_2(S_{7,3}) = |\mathcal{F}| \ge 8$, as required.
    \end{itemize}
\end{proof}

We now prove a more general result including all spiders whose legs have length at least two. 

\begin{theorem}\label{thm:spiders}
Let $S_{n,q}$ be a spider with $n$ vertices and $q\ge 3$ legs, each of length at least $2$. \\
If $q=3$ or at least two legs have length $2$, then
\[
c_2(S_{n,q})=
\begin{cases}
n-1+\lfloor 2q/3\rfloor & \text{if $q=3$ or $q\equiv 1\pmod 3$,}\\[0.2cm]
n-2+\lfloor 2q/3\rfloor & \text{otherwise.}
\end{cases}
\]
If $q>3$ and at most one leg has length $2$, then
\[
c_2(S_{n,q})=
\begin{cases}
n-1+\lfloor 2q/3\rfloor & \text{if $q\equiv 1\pmod 3$,}\\[0.2cm]
n-2+\lfloor 2q/3\rfloor+\varepsilon & \text{otherwise,}
\end{cases}
\]
where $\varepsilon\in\{0,1\}$.
\end{theorem}
\begin{proof}
We first establish the lower bound by induction on $n$, keeping $q$ fixed. The base case corresponds to the spider $S_{2q+1,q}$, for which the desired bound holds by Proposition~\ref{prop:short-spider}. Now, let $S_{n,q}$ be a spider with $n > 2q+1$ vertices, such that each of its $q$ legs has length at least two. Since $n > 2q+1$, at least one of the legs must have length at least three. Fix one such leg, and let $e$ be the edge in this leg that is incident to a leaf $v$. Let $S_{n-1,q}$ denote the spider obtained from $S_{n,q}$ by deleting the leaf $v$. Then $S_{n-1,q}$ has $n - 1$ vertices and $q$ legs, each of length at least two. 

Let $\mathcal{F}$ be a minimum $2$-RSPS for $S_{n,q}$. Observe that $\mathcal{F}$ must contain a single-edge path (of some color) consisting only of the edge $e$, since $e$ must be separated from its unique adjacent edge. Moreover, this path cannot separate any pair of edges in $S_{n-1,q}$, so we require at least $c_2(S_{n-1,q})$ additional paths to separate every pair of edges in $S_{n-1,q}$. Therefore,
\[
c_2(S_{n,q}) \ge c_2(S_{n-1,q}) + 1,
\]
which, by the induction hypothesis, yields the desired bound.

For the upper bound, we construct a $2$-RSPS inductively, following a strategy similar to the one used in Theorem~\ref{thm: paths} for $P_n$. Let $\mathcal{L}=\{\ell_1,\dots,\ell_q\}$ denote the set of legs of $S_{n,q}$. For each leg $\ell\in \mathcal{L}$, we refer to the edge incident to the leaf as the \emph{terminal edge}, denoted by $e_t^{\ell}$, and to the edge adjacent to it (towards the center) as the \emph{penultimate edge}, denoted by $e_p^{\ell}$. 

We prove, by induction on $n$ with $q$ fixed, that $S_{n,q}$ admits a $2$-RSPS of the desired size satisfying the following properties:
\begin{enumerate}[label=$(\roman*)$]
    \item \label{item:spider-1} 
    If monochromatic edges exist, then each of them is incident to the center of the spider; moreover, there are at most two, and if both exist, they have distinct colors and belong to different legs. Denote them by $e_{\operatorname{red}}$ and $e_{\operatorname{blue}}$, respectively.

     \item \label{item:spider-2} For each leg $\ell\in \mathcal{L}$ that either has length at least three or contains no monochromatic edge, 
     $e_p^{\ell}$ belongs to both an $\alpha$-path and a $\beta$-path, neither containing $e_t^\ell$, where $\{\alpha,\beta\}=\{\operatorname{red},\operatorname{blue}\}$.

    \item \label{item:spider-3} For each leg $\ell\in\mathcal{L}$, the terminal edge $e_t^\ell$ is contained in exactly two paths: a single-edge path of one color and a path of the other color that also contains $e_p^\ell$.  
\end{enumerate}

\medskip

\noindent
Let us first establish the base cases. 
\begin{itemize}
    \item If $q=3$ or at least two legs have length two, we take $n=2q+1$ as the base case. We use the $2$-RSPS constructed in the proof of Proposition~\ref{prop:short-spider}. Such construction is as follows. For $S_{7,3}$, $S_{9,4}$, $S_{11,5}$, and $S_{13,6}$, we take the corresponding configurations shown in Figure~\ref{fig:spiders_base_optimal}. We then iteratively add disjoint copies of $S_{7,3}$, colored as in Figure~\ref{fig:spiders_base_optimal}, until obtaining a $2$-RSPS for $S_{2q+1,q}$ using $2q-1+\lfloor{2q/3}\rfloor + r$ paths, where $r=1$ if $q=3$ or $q\equiv 1\pmod 3$, and $r=0$ otherwise. 

    \item Otherwise, if $q>3$ and at most one leg has length two, so that at least three legs have length at least three, let $n'=(2q+1)+t$, where $t=0$ if $q\equiv 0\pmod 3$, $t=2$ if $q\equiv 1\pmod 3$, and $t=1$ if $q\equiv 2\pmod 3$. We take $S_{n',q}$ as the base case, starting from the corresponding configuration in Figure~\ref{fig:spiders_base_modified}. We then add disjoint copies of $S_{7,3}$, colored as in Figure~\ref{fig:spiders_base_optimal}, until obtaining a $2$-RSPS for $S_{n',q}$ using $n'-1+\lfloor 2q/3 \rfloor$ paths.    

\end{itemize}

\begin{figure}[h!]
    \centering
    \includegraphics[width=1\linewidth]{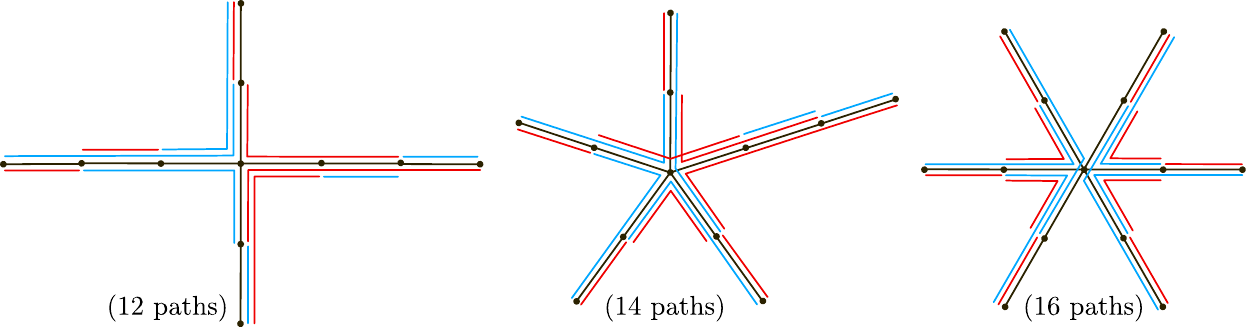}
    \caption{$2$-RSPS's for the spiders $S_{11,4}$, $S_{12,5}$, and $S_{13,6}$, each using $n-1+\lfloor 2q/3\rfloor$ paths. These are used as base configurations in the case $q>3$ when at most one leg has length $2$.}
    \label{fig:spiders_base_modified}
\end{figure}

It is straightforward to verify that the separating systems constructed in the base cases satisfy properties~\ref{item:spider-1}--\ref{item:spider-3}, and have the desired size.

\medskip

Now fix $S_{n,q}$ with $n$ larger than the base case, and choose a leg $\ell$ of length at least three. Let $S_{n-1,q}$ be obtained from $S_{n,q}$ by deleting the terminal edge $e^*$ of $\ell$. Let $\ell'$ denote the leg in $S_{n-1,q}$ corresponding to $\ell$, and let $e_t$ and $e_p$ denote the terminal and penultimate edges of $\ell'$ in $S_{n-1,q}$, so $e_t$ is incident to $e^*$. By the induction hypothesis, there exists a $2$-RSPS $\mathcal{F}_{n-1,q}$ of the correct size (as in the statement, depending on parity mod $3$) and satisfying properties~\ref{item:spider-1}--\ref{item:spider-3}. Moreover, by property~\ref{item:spider-1}, there are at most two legs of $S_{n-1,q}$ containing a monochromatic edge. Hence, after relabeling the legs if necessary, we may assume that:
\begin{itemize}
    \item if $q=3$ or $S_{n,q}$ has at least two legs of length $2$, then the legs containing monochromatic edges in $\mathcal{F}_{n-1,q}$ (if they exist) are legs of length $2$, and in particular $\ell'$ contains no monochromatic edge; and
    \item if $q>3$ and $S_{n,q}$ has at most one leg of length $2$, then either $\ell'$ contains no monochromatic edge in $\mathcal{F}_{n-1,q}$, or $\ell'$ has length at least three. 
\end{itemize}

Therefore, property~\ref{item:spider-2} applies to $\ell'$ in $\mathcal{F}_{n-1,q}$, thus implying that $e_p$ belongs to both an $\alpha$-path and a $\beta$-path, neither containing $e_t$, where $\{\alpha,\beta\}=\{\operatorname{red},\operatorname{blue}\}$.

\medskip
We extend $\mathcal{F}_{n-1,q}$ to a path system $\mathcal{F}_{n,q}$ as follows:
\begin{enumerate}
    \item Extend the single-edge $\alpha$-path containing $e_t$ to include $e^*$.
    \item Extend an $\alpha$-path containing $e_p$ but not $e_t$ to include $e_t$.
    \item Include all the remaining paths from $\mathcal{F}_{n-1,q}$ in $\mathcal{F}_{n,q}$ without modification. 
    \item Add a new single-edge $\beta$-path consisting of $e^*$.
\end{enumerate}

Steps~1 and~4 ensure that $e^*$ is bicolored. Since no  edge in $S_{n-1,q}$ changes its color status, the set of monochromatic edges is the same as in $\mathcal{F}_{n-1,q}$. Hence property~\ref{item:spider-1} is preserved. 
We next verify properties~\ref{item:spider-2} and~\ref{item:spider-3} for the chosen leg in $S_{n,q}$. The new terminal edge is $e^*$, which belongs to a single-edge $\beta$-path by Step~4 and to an $\alpha$-path containing $e_t$ by Step~1 and to no other paths. Thus property~\ref{item:spider-3} holds for $e^*$. The new penultimate edge is $e_t$, which belongs to an $\alpha$-path not containing $e^*$ by Step~2. Also, since $e_t$ is not incident to the center, property~\ref{item:spider-1} implies that it is bicolored; hence it belongs to a $\beta$-path as well, and this path is unchanged and does not contain $e^*$. Therefore property~\ref{item:spider-2} holds for the chosen leg. The remaining legs are unaffected, so properties~\ref{item:spider-2} and~\ref{item:spider-3} are preserved. 

\medskip
It remains to verify that $\mathcal{F}_{n,q}$ is a separating system. Observe that exactly two paths of $\mathcal{F}_{n-1,q}$ were extended, both of color $\alpha$, and one new single-edge $\beta$-path was added. 
\begin{itemize}
    \item First, consider two edges of $S_{n-1,q}$. Any pair distinct from $\{e_t,e_p\}$ remains separated. Indeed, the only modified paths are two paths on the chosen leg, and they are extended only by adding edges, so no separation between edges of $S_{n-1,q}$ is lost except possibly for the pair $\{e_t,e_p\}$, which we check separately. For the pair $\{e_t,e_p\}$, property~\ref{item:spider-2} applied to $\mathcal{F}_{n-1,q}$ yields a $\beta$-path containing $e_p$ but not $e_t$, and property~\ref{item:spider-3} yields a single-edge $\alpha$-path containing $e_t$. The former is unchanged, and the latter is extended only by adding $e^*$, so it still does not contain $e_p$. Thus $e_t$ and $e_p$ remain separated.

    \item Next, the edges $e^*$ and $e_t$ are separated. Indeed, by property~\ref{item:spider-3} applied to $\mathcal{F}_{n,q}$, $e^*$ lies in a single-edge $\beta$-path, which clearly does not contain $e_t$, while by property~\ref{item:spider-2} applied to $\mathcal{F}_{n,q}$, $e_t$ lies in an $\alpha$-path not containing $e^*$, with $\{\alpha,\beta\}=\{\operatorname{red},\operatorname{blue}\}$. Hence $e^*$ and $e_t$ are separated.

    \item Finally, let $e$ be an edge in $S_{n-1,q}$ with $e\neq e_t$. Since $\mathcal{F}_{n-1,q}$ is a separating system, there exists a path in $\mathcal{F}_{n-1,q}$ that contains $e$ but not $e_t$. Such a path is either unchanged in $\mathcal{F}_{n,q}$, or it is extended to include $e_t$. In any case, such a path does not contain $e^*$. This separates $e$ from $e^*$. Conversely, by Steps~1 and~4, the edge $e^*$ is contained in two paths of different colors: a single-edge $\beta$-path consisting only of $e^*$, and an $\alpha$-path containing only $e_t$ and $e^*$. Since $e\neq e_t$,  both of these paths avoid $e$. Hence $e^*$ is also separated from $e$.
\end{itemize}
Therefore, $\mathcal{F}_{n,q}$ separates every pair of edges of $S_{n,q}$, and so it is a $2$-RSPS of $S_{n,q}$.

\medskip
Since each inductive step adds exactly one path, together with the corresponding base case we obtain the claimed bounds, completing the proof.  
\end{proof}

\begin{remark}
    Curiously, there exist spiders with five or six legs for which the lower bound of Theorem \ref{thm:spiders} is not tight. For the sake of completeness, we construct such examples in Appendix \ref{lowerbound}. 
\end{remark}

%


\subsection{Trees}

In this final subsection regarding $2$-RSPS's, we will prove that every tree $T$ on $n$ vertices has $2$-rainbow separation number $c_2(T)\leq\frac{4(n-1)}{3}$.

In order to guarantee sufficient color redundancy in our constructions for trees, we will often require that every edge is covered by at least one red path and one blue path. This motivates the following definition.
We say that a $2$-RSPS (for a graph $G$) satisfies the \textit{bicolor edge-coverage property} if every edge of $G$ belongs to at least one red path and one blue path in the system. We will occasionally abuse terminology and say that $G$ has the bicolor edge-coverage property (if the implied $2$-RSPS is clear from context). 

In the following result, we consider three specific types of trees that will be useful in the construction of a $2$-RSPS for general trees.

\begin{obs}\label{observation}
Let $n\ge 7$. For the star $S_4$, the path $P_n$, and the spider $S_{n,3}$ with three legs of length at least $2$, we have: 
    \begin{enumerate}[label=$(\roman*)$]
    \item \label{item:star} $c_2(S_4)=4 =\frac{4(|S_4|-1)}{3}$, and there exists an optimal $2$-RSPS for $S_4$ satisfying the bicolor edge-coverage property; see the construction in Figure~\ref{fig:stars}.
    
    \item \label{item:path} $c_2(P_{n})=n+1\le \frac{4(|P_n|-1)}{3}$, and there exists an optimal $2$-RSPS for $P_n$ that satisfies the bicolor edge-coverage property; see the construction in the proof of Theorem~\ref{thm: paths}. 
    
    \item \label{item:spider} $c_2(S_{n,3})\le n+1\le \frac{4(|S_{n,3}|-1)}{3}$, and there exists an optimal $2$-RSPS for $S_{n,3}$ satisfying the bicolor edge-coverage property; see the construction given in the proof of Theorem~\ref{thm:spiders}.
    \end{enumerate}
\end{obs}

We now introduce a notion on spiders that will be used to establish a tight upper bound for $c_2(T)$, where $T$ is a tree.  
Recall that both stars and paths are examples of spiders. 
Given a tree $T$ and a spider $S$ contained in $T$, we say that $S$ is a \textit{bare spider}  if the following conditions hold:
\begin{itemize}
    \item the subgraph of $T$ induced by $V(T)\setminus V(S)$ is connected;  
    \item either $S=T$, or $T$ is not a spider and the vertex $h$ that separates $S$ from the rest of $T$ is the unique vertex of $S$ with degree at least $3$ in $T$. 
\end{itemize} 
We refer to $h$ as the \textit{spider head}, and to its unique neighbor $v$ outside of $S$ (if it exists) as the \textit{attachment vertex} of the bare spider. The edge $\{h,v\}$ is the \emph{hat} of $S$. We define the \textit{extended bare spider} to be $S+v$. When $S=T$ and $T$ is not a path, we define the spider head as the unique vertex of degree at least $3$, whereas if $T$ is a path, then every vertex can be considered a spider head.

 It is easy to see that the bare spiders of a tree are pairwise vertex-disjoint. Another simple observation is the following: 

\begin{equation}
    \label{eq:two_bare_spiders}
    \textrm{if $T$ is not a spider itself, then it contains at least two bare spiders.}
\end{equation}

Indeed, this can be proved by induction. Let $T$ be a non-spider and delete from it a leaf $v$. If $T-v$ is a spider and $v$ is adjacent to a spider head or leaf of $T-v$, then $T$ is a spider, giving a contradiction. If $T-v$ is a spider and $v$ is adjacent to some other vertex of $T-v$, then $T$ contains exactly two bare spiders. If $T-v$ is not a spider, then it contains two bare spiders by induction, say $S_1$ and $S_2$. Both of these remain such in $T$ (possibly with an extra leg if $v$ is adjacent to the spider head of either $S_1$ or $S_2$) unless $v$ is adjacent to an internal non-head vertex of, say, $S_2$. But then the neighbor $u$ of $v$ is the head of a bare spider $S_3$ with two legs, so $S_1$ and $S_3$ are two bare spiders of $T$.

We now establish an upper bound on $c_2(T)$ for any tree $T$. 
\begin{theorem}\label{tree2}
    Let $T$ be a tree with $n\ge 2$ vertices. Then,  
    \[c_2(T)\le \frac{4(n-1)}{3}. \]
\end{theorem}
\begin{proof}
We construct a $2$-RSPS for $T$ using at most $\frac{4(n - 1)}{3}$ paths.
We proceed iteratively, constructing such a system for a bare spider (or extended bare spider) of $T$, removing said spider from $T$, and continuing this process in the remainder tree. To manage the process, we initialize $T' := T$, and we update $T'$ at each step by removing the edges already covered by the constructed paths.

At each step, we select a bare spider $S$ of $T'$ with spider head $h$. If $S\ne T$, let $v$ denote the attachment vertex of $S$. Our plan is to construct a $2$-RSPS $\mathcal{F}_S$ (or $\mathcal{F}_{S+v}$) for $S$ (or for its extension $S+v$), ensuring that $|\mathcal{F}_S| \le \frac{4\cdot e(S)}{3}$ (or $|\mathcal{F}_{S+v}|\le \frac{4\cdot e(S+v)}{3}$) and that the bicolor edge-coverage property is satisfied. This will always be possible, except in the special case of ``bad spiders", as we shall see below. After constructing the corresponding path system, we will remove the internal edges of $S$ (or $S+v$) from $T’$ and repeat the process on the updated tree. By the first property of bare spiders, the remaining subgraph will remain connected after each deletion, thus guaranteeing that the iterative process continues. We will ensure that each $2$-RSPS that we construct for an (extended) bare spider satisfies the bicolor edge-coverage property, with a possible exception at the last step, as we shall explain. 

\paragraph{Construction of a $2$-RSPS $\mathcal{F}_S$ for a fixed bare spider $S$ (in $T'$):}
We initialize $S' := S$, and iteratively remove edges from $S'$ as we construct the path system. Set $\mathcal{F}_S := \emptyset$.

We first apply the following reductions to $S'$, removing the involved edges at each step:

\begin{enumerate}
    \item While there is a set of three leaves adjacent to $h$, we use the construction from Observation~\ref{observation}~\ref{item:star} to build a $2$-RSPS of four paths satisfying the bicolor edge-coverage property. We add these paths to $\mathcal{F}_S$ and remove the three edges from $S'$ and $T'$. 
    
    \item While there is a subspider $S_{n',3}$ in $S'$ with $3$ legs, all of length at least two and terminating in leaves, we apply Observation~\ref{observation}~\ref{item:spider} to construct a $2$-RSPS of size at most $\frac{4\cdot e(S_{n',3})}{3}$ that satisfies the bicolor edge-coverage property. We add the paths to $\mathcal{F}_S$ and we remove the edges of this subspider from $S'$ and $T'$.
    
    \item While there is a path $P_{n'}$ in $S'$ with $n' \ge 7$ vertices, we apply Observation~\ref{observation}~\ref{item:path} to construct a $2$-RSPS of size at most $\frac{4\cdot e(P_{n'})}{3}$ that satisfies the bicolor edge-coverage property. The vertex $h$ may appear as an endpoint or as an internal vertex of $P_{n'}$. We add the constructed paths to $\mathcal{F}_S$ and remove the edges of this path from $S'$ and $T'$. 
\end{enumerate}

After these reductions, one of the following occurs:
\begin{enumerate}[label=$(\alph*)$]
    \item All edges of $S$ have been covered, in which case we proceed to another bare spider, or
    \item   A residual spider $S'$ remains, satisfying:
    \begin{itemize}
        \item $h$ is adjacent to at most two leaves, 
        \item at most two legs have length $2$ or more,
        \item each leg has length at most $5$,
        \item the sum of the lengths of any two legs is at most $5$. 
    \end{itemize}
\end{enumerate} 
Suppose that the latter occurs. Let $L_{S'}$ denote the vector of leg lengths of $S'$, sorted in non-increasing order. Then: 
\begin{align*}
L_{S'} \in \{ &(5),\ (4,1,1),\ (4,1),\ (4),\ (3,2,1,1),\ (3,2,1),\ (3,2),\ (3,1,1),\ \\
              &(3,1),\ (3),\ (2,2,1,1),\ (2,2,1),\ (2,2),\ (2,1,1),\ (2,1),\ (2),\ (1,1),\ (1) \}.
\end{align*}

We now examine each case:

        \begin{figure}[h!]
    \centering
    \includegraphics[width=0.9\textwidth]{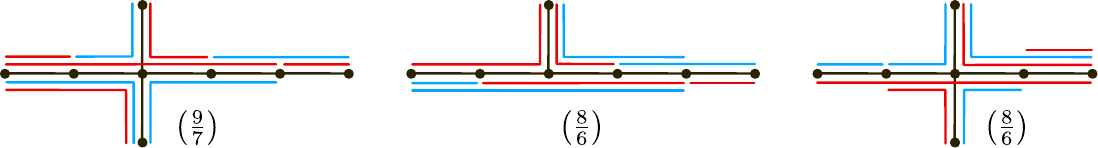}
    \caption{$2$-RSPS's for spiders whose legs have lengths $(3,2,1,1)$, $(3,2,1)$ or $(2,2,1,1)$. }
    \label{fig:tree_1}
\end{figure}

    \begin{figure}[h!] 
    \centering
    \includegraphics[width=1\textwidth]{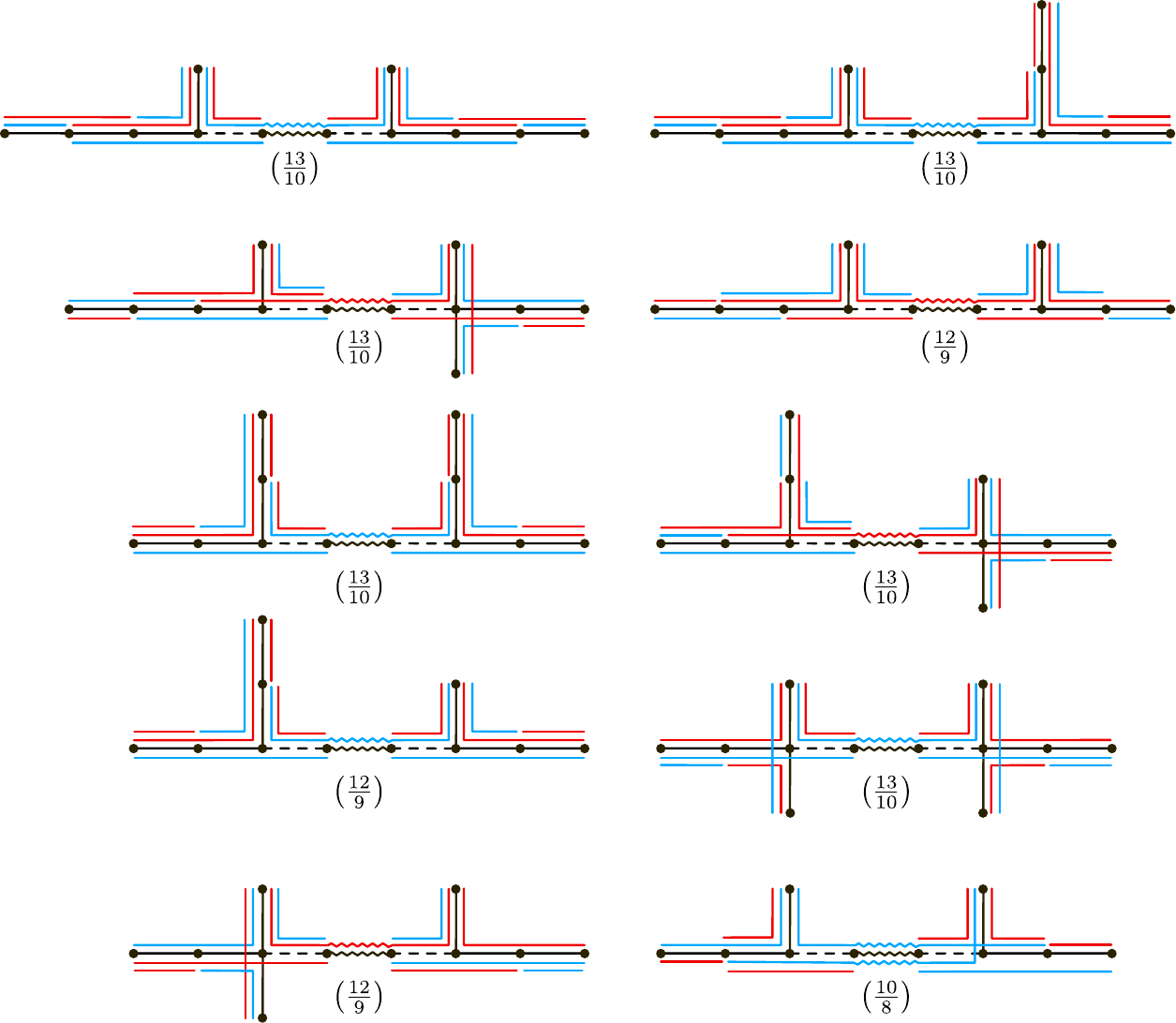}
    \caption{$2$-RSPS's for pairs of bad spiders, where each bad spider is of the form $S'+v$ with $L_{S'} \in \{(3,1), (2,2), (2,1,1), (2,1)\}$. The dashed edges represent the hat $hv$ contained in each bad spider, while the zigzag edge represents the rest of the path in $T$ that connects the two spider heads.}
    \label{fig:bad_spiders}
\end{figure}

      \begin{figure}[h!]
    \centering
    \includegraphics[width=0.9\textwidth]{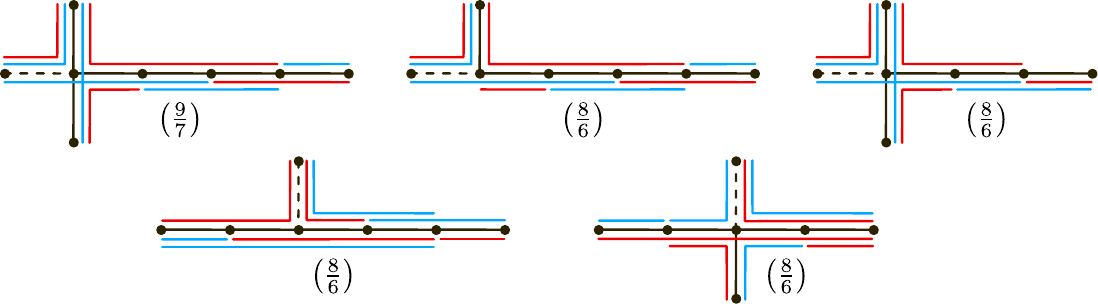}
    \caption{$2$-RSPS's for bare spiders whose legs have lengths $(4,1,1)$, $(4,1)$, $(3,2)$, $(3,1,1)$ or $(2,2,1)$. In each case, we consider the extended bare spider $S^*$, where the hat $hv$ is drawn with a dashed line.}
    \label{fig:tree_2}
\end{figure}

\begin{itemize}

     \item If $S^{'} \ne T'$ and $L_{S'} \in \{(5), (4), (3), (2), (1)\}$, we defer the treatment of $S'$ to a later step, where it will be part of another bare spider.
     
    \item If $L_{S'} \in \{(3,2,1,1), (3,2,1), (2,2,1,1)\}$, we use the constructions shown in Figure~\ref{fig:tree_1}, which satisfy bicolor edge-coverage.

    \item If $S^{'}\ne T'$ and $L_{S'} \in \{(4,1,1), (4,1), (3,2), (3,1,1), (2,2,1), (1,1)\}$, we consider the extended spider $S^*:=S'+v$ and we apply the corresponding construction shown in Figure~\ref{fig:tree_2} (or, for $(1,1)$, Figure~\ref{fig:stars}), which satisfies bicolor edge-coverage.

      \item If $S^{'}\ne T'$ and $L_{S'} \in \{(3,1), (2,2), (2,1,1), (2,1)\}$, we refer to the extended spider $S^*:=S'+v$ as a \textit{bad spider}. We pair bad spiders and we apply the corresponding construction shown in Figure~\ref{fig:bad_spiders} to each pair, or shown in Figure~\ref{fig:shared_hat_bad_spiders} if the hat is shared. We then remove from $T'$ the two paired bad spiders (but not the rest of the connecting path), and we proceed to the next step. Bicolor edge-coverage is satisfied by all of our constructions from Figure \ref{fig:bad_spiders}. \textbf{As for our constructions from Figure \ref{fig:shared_hat_bad_spiders}, each violates the bicolor edge-coverage property at most once, except when the leg vectors are $(2,2)$ and $(2,1)$ or $(2,1,1)$ and $(2,1)$, in which case it is violated twice in distinct colors.} By (\ref{eq:two_bare_spiders}), we can always either choose $S$ so that $S^*$ is not a bad spider, or select two bad spiders to pair. It is important to note that two bad spiders $S_1^*$ and $S_2^*$ can only share a hat if together (and with their hat) they are the entirety of $T'$; indeed, if $T'\setminus (S_1^*\cup S_2^*)$, is not empty, then it is attached to one of the two spider heads, so deleting the corresponding bare spider disconnects $T'$, a contradiction.  
    
     \item Finally, it may be that $S' = T'$ and $L_{S'}\notin \{(3,2,1,1),(3,2,1),(2,2,1,1)\}$. This case, if encountered, is unique, as it corresponds to the last step of the decomposition. If $S'$ is a path, then we apply the construction shown in Figure~\ref{fig:paths-upper}, except when $S'$ is a single edge, in which case we simply take a single-edge path (of any color) covering it. \textbf{In this case, the bicolor edge-coverage property fails at most once, except when $S'\in\{P_3, P_5, P_6\}$, in which case it is violated twice in distinct colors; see Figure~\ref{fig:paths-upper}.} Otherwise, we have $L_{S'}\in \{(4,1,1),(3,1,1),(2,2,1),(2,1,1)\}$. In that case, we construct four paths for the three edges incident to the spider head of $S'$, as in~\ref{item:star}. The remaining edges form either a path of length at most $3$, or two disjoint edges. In the former case, we construct a $2$-RSPS for the path as in the previous case; in the latter, we construct two single-edge paths of different colors covering the two disjoint edges. See Figure \ref{fig:small_bad_spiders}. \textbf{The bicolor edge-coverage property is violated at most once in the cases $(4,1,1)$ and $(2,1,1)$, but it is violated twice in distinct colors in the cases $(3,1,1)$ (due to the ensuing $P_3$) and $(2,2,1)$ (due to the ensuing two disjoint edges).}
     \end{itemize}

     In each step, we obtain a $2$-RSPS of size at most $\frac{4}{3}$ the number of deleted edges that satisfies the bicolor edge-coverage property (with the noted exceptions). That is, each step contributes at most $\frac{4m}{3}$ paths, where $m$ is the number of edges involved. So, the total number of paths in the union $\mathcal{F}$ of these $2$-RSPS's is at most $\frac{4(n - 1)}{3}$, as desired.  

    To complete the proof, there is a final detail that we need to verify in order to ensure that $\mathcal{F}$ is itself a $2$-RSPS. Suppose that, at some step, we pair two bad spiders $S_1^*$ and $S_2^*$ that have distinct hats connected through a path $P$ in $T'\setminus(S_1^*\cup S_2^*)$. If $e\in E(T'\setminus (S_1^*\cup S_2^*))$, then we cannot use the paths of Figure \ref{fig:bad_spiders} that intersect $P$ (let us call those \emph{bridging paths}) to separate $E(S_1^*\cup S_2^*)$ from $e$.

    If our $2$-RSPS for $S_1^*\cup S_2^*$ does not have any monochromatic edges even with the bridging paths removed, then again it is straightforward to separate $E(S_1^*\cup S_2^*)$ from $E(S^\dagger)$. The reader may examine all the pictures in Figure~\ref{fig:bad_spiders} to see that, indeed, in each case, even with the bridging paths removed, there are no monochromatic edges, with just one exception: the case $L_{S_1}=L_{S_2}=(2,1)$ (the relevant picture is at the bottom right corner of Figure~\ref{fig:bad_spiders}). In this case, when the bridging paths are removed, we get exactly two monochromatic edges, both in red (of course, these edges are separated from each other in the original construction with the help of one of the bridging paths). So, this issue only affects pairs of bad spiders with $L_{S_1}=L_{S_2}=(2,1)$.
    
    As long as the (extended) bare spider (or pair of connected bad spiders) of the decomposition that contains $e$, say $S^\dagger$, satisfies the bicolor edge-coverage property, or contains only one monochromatic edge in blue, we can separate each edge $e'\in E(S_1^*\cup S_2^*)$ from $e$ with a non-bridging path, and separate $e$ from $e'$ with a path of the opposite color contained in the $2$-RSPS of $S^\dagger$. As we have already seen (see the passages in \textbf{bold}), the only case in which bicolor edge-coverage can be violated in $S^\dagger$ is if $S^\dagger=T'$. Specifically, in that case, $S^\dagger$ violates the bicolor edge-coverage property at most once (in blue) unless it is one of $P_3$, $P_5$, $P_6$, or has $L_{S^\dagger}=(3,1,1)$ or $(2,2,1)$, or is the union of two bad spiders with shared hat and leg vectors $(2,2)$ and $(2,1)$ or $(2,1,1)$ and $(2,1)$. In all of these cases, $S^\dagger$ violates the bicolor edge-coverage property twice, once in each color. So, for the rest of the proof, we assume that $S^\dagger$ is the final spider of the decomposition and that it has one of the above-listed structures.

    \begin{figure}[h!] 
    \centering
    \includegraphics[width=0.8\textwidth]{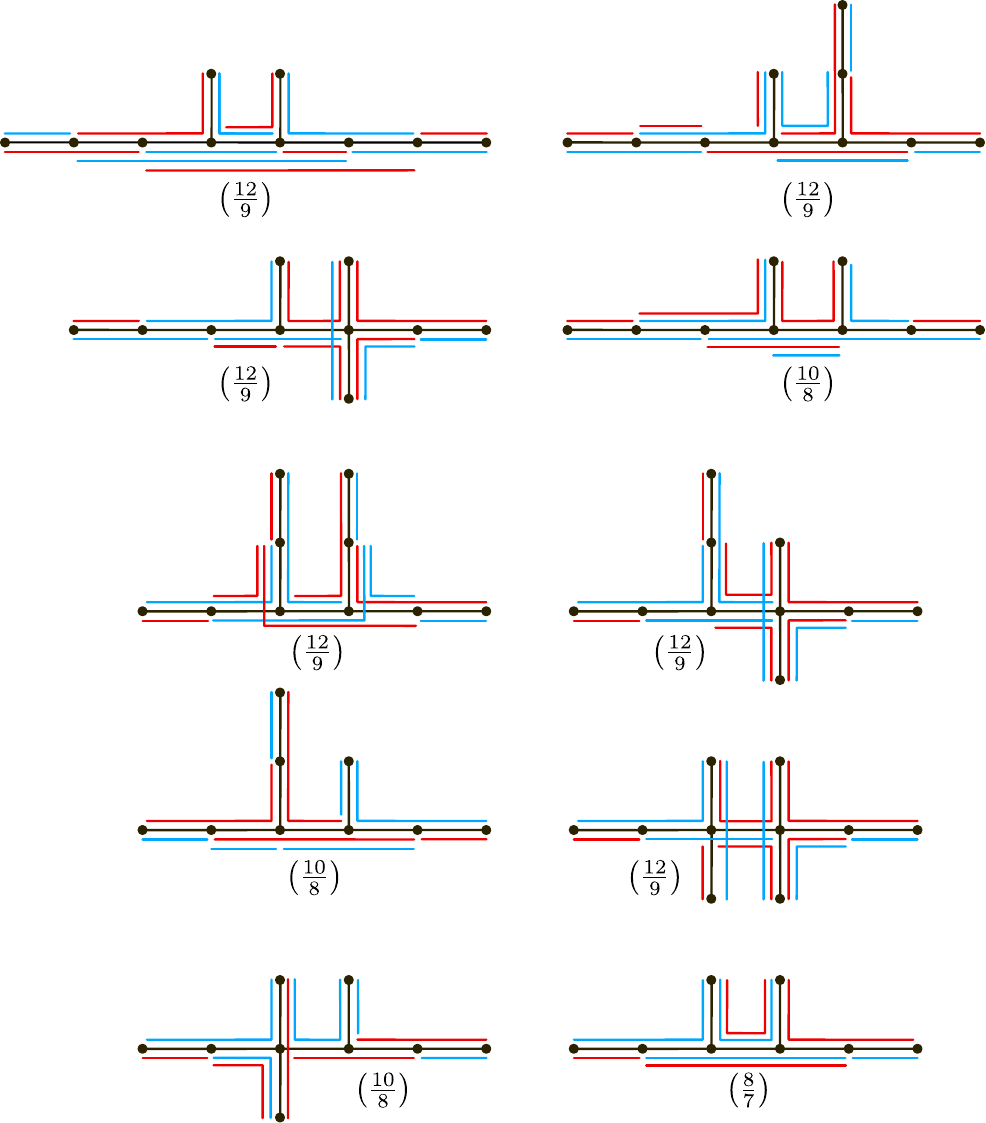}
    \caption{$2$-RSPS's for pairs of bad spiders sharing the same hat, where each bad spider $S'$ satisfies $L_{S'} \in \{(3,1),(2,2),(2,1,1),(2,1)\}$.} 
    \label{fig:shared_hat_bad_spiders}
\end{figure}

\begin{figure}[h!] 
    \centering
    \includegraphics[width=0.9\textwidth]{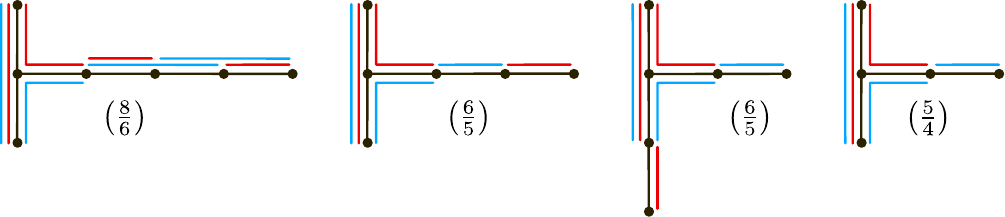}
    \caption{$2$-RSPS's for the exceptional final bare spiders.} 
    \label{fig:small_bad_spiders}
\end{figure}
    
     To resolve this problem, we simply need to add one blue single-edge path to the $2$-RSPS of $S^\dagger$ in order to ensure that it yields at most one monochromatic (blue) edge. We then couple $S^\dagger$ with a pair of bad spiders $S_1^*$ and $S_2^*$ such that $L_{S_1}=L_{S_2}=(2,1)$, and show that we have used at most $\frac{4\cdot e(S_1^*\cup S_2^*\cup S^\dagger)}{3}$ paths to separate $S_1^*\cup S_2^*\cup S^\dagger$. We conclude with a table that demonstrates this.

\begin{table}[h!]
\begin{center}
\begin{tabular}{ c | c | c }
$S^\dagger$ (or $L_{S^\dagger}$) & $e(S_1^*\cup S_2^*\cup S^\dagger)$ & size of $2$-RSPS for $S_1^*\cup S_2^*\cup S^\dagger$\\
\hline
$P_3$ & $10$ & $13$\\
$P_5$ & $12$ & $16$\\
$P_6$ & $13$ & $17$\\
$(3,1,1)$ & $13$ & $17$\\
$(2,2,1)$ & $13$ & $17$\\
$(2,2)+(2,1)$&16&21\\
$(2,1,1)+(2,1)$&16&21\\

\end{tabular}
\end{center}
\end{table}

\end{proof}

 \newpage

\section{More colors}\label{sec3}

For $k\geq 3$, we begin by determining $c_k$ for paths, cycles and stars. We will see that, for each of these classes of graphs, three colors already suffice to achieve the strong separation number. 

\begin{theorem}\label{3colpath}
    For the path $P_n$, we have $c_k(P_n)=n-1$ for $n\leq k+1$ and $c_k(P_n)=n$ otherwise, for every $k\geq 3$. 
\end{theorem}

\begin{proof}

    For $n\leq k+1$, we simply recall that $c_k(P_n)\geq ssp(P_n)$ (lower bound) and then we color each edge a distinct color (upper bound), so let us focus on the case $n>k+1$. 

    For the upper bound, only three colors are required. Let $c_0$:=red, $c_1$:=blue and $c_2$:=green, where indices are taken mod $3$. We construct the $3$-RSPS $\mathcal{F}_n$ by adding to it the following:
    \begin{enumerate}
        \item A red path containing only $e_1$.
        \item A $c_k$-path containing exactly $e_k$ and $e_{k+1}$ for each $k\leq n-2$.
        \item A $c_{n-1}$ path containing only $e_{n-1}$.
    \end{enumerate}

    See Figure \ref{path_3col} for the construction for $P_7$.

        \begin{figure}[h!]
    \centering
    \includegraphics[width=0.35\linewidth]{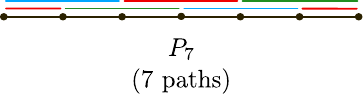}
    \caption{Optimal $3$-RSPS for $P_7$.} 
    \label{path_3col}
\end{figure}

    Every edge of $P_n$ is in two paths of different colors of $\mathcal{F}_n$. If none of these paths coincide for two edges $e_i$ and $e_j$, then they are separated. Otherwise, $e_i$ and $e_j$ are consecutive. But then the two paths ending at their common vertex have different colors, and so separate $e_i$ and $e_j$.

    For the lower bound, we first recall that every RSPS is a strongly separating path system, and we note that the only strongly separating path system of $P_n$ that consists of at most $n-1$ paths is $E(P_n)$. This is proven by induction. It is true for $P_3$. Suppose that it also holds for $P_n$. A strongly separating path system of $P_{n+1}$ must contain a path that has only the edge $e_n$. Suppose that said system has size at most $n$, then $P_{n+1}\setminus\{ e_n\}$ is separated by $n-1$ paths. By induction, these paths are exactly the elements of $E(P_{n+1})\setminus \{e_n\}$, implying that the entire system is $E(P_{n+1})$.

    Since $n-1>k$, there is no way to color $E(P_n)$ in such a way that it rainbow separates the edges of $P_n$. We conclude that $c_k(P_n)>ssp(P_n)\Rightarrow c_k(P_n)\geq n$. 
    
\end{proof}

\begin{theorem}
    For the cycle $C_n$, we have $c_k(C_n)=n$ for every $k\geq 3$.
\end{theorem}

\begin{proof}

    The lower bound follows from the fact that $c_k(C_n)\geq ssp(C_n)=n$. For $C_3$, the upper bound comes from single-edge paths of distinct colors. Let us prove the upper bound for $n\geq 4$ by an explicit construction. Once again, three colors are sufficient.

    If $n\equiv 0,1\text{ mod }3$, we add to $\mathcal{F}_n$ all the paths of length $2$, cyclically alternating the colors. Of course, if $n\equiv 1\text{ mod }3$, then at the point where we ``close the cycle" there will be a single instance in which a color will be used twice in a row. For each edge in any pair of non-consecutive edges (except for a pair including the monochromatic edge in the case $n\equiv 1\text{ mod }3$), there are two paths of distinct colors that contain that edge and avoid the other edge of the pair, so such pairs of edges are separated. In the case where $n\equiv 1\text{ mod }3$ and one edge in the pair is the unique monochromatic edge, say covered by two green paths, these do not cover the second edge of the pair, while there are two paths of different colors, at least one of which is not green, covering the second edge of the pair, but not the monochromatic edge. Thus, all the pairs of non-consecutive edges are separated. Moreover, for every two consecutive edges, the two paths that end at their common vertex have distinct colors, so consecutive edges are also separated. We conclude that $\mathcal{F}_n$ is a $3$-RSPS.

   For $n\equiv 2\text{ mod }3$, if we take all the paths of length $2$ with cyclically alternating colors, then the penultimate path is the same color as the first path, so we will get two paths of the same color that share a vertex. So, we instead use the paths of length $4$. We make $\mathcal{F}_n$ consist of all the paths of length $4$, cyclically alternating the colors; at the point where we ``close the cycle", there is a single instance in which one color is skipped. Then the path of length $4$ starting with the edge $e_{i+1}$ has the next color (in the cyclic order) from the path of length $4$ ending with $e_i$, or, if the single instance in which one color is skipped occurs between $e_{i+1}$ and $e_{i-3}$, it has the next to next color. Hence, for every two consecutive edges, the two paths that end at their common vertex have distinct colors, separating said edges. Moreover, if $n>5$, it is true that for each edge in any pair of non-consecutive edges, there are two paths of distinct colors that contain it and avoid the other edge of the pair, so non-consecutive edges are also separated. 
   
   Examples for $C_6$, $C_7$ and $C_8$ can be found in Figure \ref{cycle_3color}. For the special case of $C_5$, see also Figure \ref{cycle_3color}.
\end{proof}

\begin{figure}[h!]
    \centering
    \includegraphics[width=0.9\linewidth]{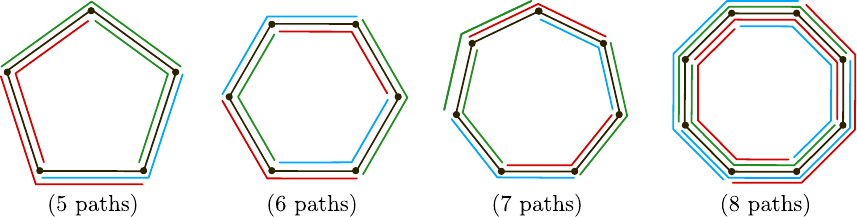}
    \caption{Optimal $3$-RSPS's for $C_5$, $C_6$, $C_7$ and $C_8$.} 
    \label{cycle_3color}
\end{figure}


\begin{theorem}\label{star3}
    For the star $S_n$ we have $c_k(S_n)=n-1$ for every $k\geq 3$.
\end{theorem}

\begin{proof}
    The lower bound follows from $c_k(S_n)\geq ssp(S_n)=n-1$. For the upper bound, only three colors are required. We firstly superimpose $\lfloor\frac{n-1}{3}\rfloor$ copies of $S_4$ separated as in Figure \ref{star_3col}. After this, up to two edges may be unaccounted for. We cover each such edge with a single-edge path of a different color. 
\end{proof}

\begin{figure}
    \centering
    \includegraphics[width=0.12\linewidth]{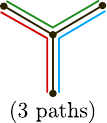}
    \caption{The building block of $3$-RSPS's of $S_n$.} 
    \label{star_3col}
\end{figure}

We now turn our attention to spiders and other trees, where the situation is different; three colors are insufficient to reach the strong separation number. 

\begin{theorem}
Let $S_{n,q}$ be a spider with $q$ legs, all of length two, so that $n=2q+1$. Then $c_3(S_{n,q})\ge 17(n-1)/16-3$.
\end{theorem}
\begin{proof}
    In an RSPS, there is at most one monochromatic edge of each color, so adding up to $3$ additional single-edge paths will yield a $3$-RSPS in which every edge receives at least two colors. We will show that such a system, $\mathcal{F}$, contains at least $17(n-1)/16$ paths. We may assume that no two copies of the same path are the same color.
    Let $A,B,C,D$ denote the number of paths of length $1,2,3,4$ respectively in our $3$-RSPS $\mathcal{F}$. Thus, $|\mathcal{F}|=A+B+C+D$. Let $x_i$ denote the number of internal edges contained in exactly $i$ paths. Let $y_i$ denote the number of external edges contained in exactly $i$ paths.
    Note that we may assume that $i\ge 2$ since each edge is covered by paths of at least $2$ colors. Then $A+2B+3C+4D=\sum_{i=2}^{\infty}i(x_i+y_i)$.

    For any path, it contains at most two edges contained in only one other path. Otherwise, there is either an edge only contained in two paths of the same color or there are two edges who share a path and whose second path is the same color. Neither of these is permitted. Thus a path of length $3$ contains at least one edge that is in at least three paths and a path of length $4$ contains at least two such edges. This gives at least $C+2D$ pairs $(p,e)$ where $p$ is a path of length at least $3$ and $e$ is an edge in that path contained in at least $3$ paths. Of the outer edges contained in at least three paths, suppose that $a$ are contained in one path of length $1$, $b$ are contained in two paths of length $1$, and $c$ are contained in three paths of length $1$. Thus the number of such pairs $(p,e)$ is at most $\sum_{i=3}^{\infty}ix_i+\sum_{i=3}^\infty iy_i-a-2b-3c$. Thus, \begin{align*}
        \sum_{i=3}^{\infty}ix_i+\sum_{i=3}^\infty iy_i-a-2b-3c&\ge C+2D\\
    \sum_{i=3}^{\infty}i(x_i+y_i)&\ge C+2D+a+2b+3c\\
    A+2B+3C+4D-2(x_2+y_2)&\ge C+2D+a+2b+3c\\
    A+2B+2C+2D&\ge 2(x_2+y_2)+a+2b+3c\\
    2|\mathcal{F}|&\ge A+2(x_2+y_2)+a+2b+3c.
    \end{align*}

Note that $a+b+c=\sum_{i=3}^{\infty}y_i$ and $A\ge (n-1)/2$ since each outer edge is in at least one, and no more than three, paths of length $1$. Thus,

\begin{align*}
    2|\mathcal{F}|&\ge (n-1)/2+2x_2+2y_2+2\sum_{i=3}^{\infty}y_i+c-a\\
    |\mathcal{F}|&\ge (n-1)/4+x_2+\sum_{i=2}^{\infty}y_i+(c-a)/2\\
    |\mathcal{F}|&\ge 3(n-1)/4+x_2-a/2.
\end{align*}

Recall that $a$ counts the number of external edges that are in at least three paths but only one of length one. These are necessarily in two paths that contain the adjacent internal edge. However, that internal edge must also be in at least one path not containing the external edge, so the internal edge is contained in at least three paths and thus not enumerated by $x_2$. Therefore, $a+x_2\le (n-1)/2$, so $-a/2\ge x_2/2-(n-1)/4$, so 
\[|\mathcal{F}|\ge 3(n-1)/4+x_2+x_2/2-(n-1)/4=(n-1)/2+3x_2/2.\]

If $x_2\ge 3(n-1)/8$, then we have the desired result of $|\mathcal{F}|\ge 17(n-1)/16$. So we may assume that $x_2<3(n-1)/8$ and thus $\sum_{i=3}^{\infty}x_i>(n-1)/8$.

Recall that the $(n-1)/2$ necessary outer paths of length $1$ do not help separate the inner star on $(n-1)/2$ edges. Any path contributes at most two edges to the inner star and the edges of the inner star are used a total of $\sum_{i=2}^{\infty}ix_i\ge 2x_2+3\sum_{i=3}^{\infty}x_i$ times. Thus,
\[|\mathcal{F}|\ge (n-1)/2+\frac{2x_2+3\sum_{i=3}^{\infty}x_i}{2}=(n-1)/2+\sum_{i=2}^{\infty}x_i+\frac{\sum_{i=3}^{\infty}x_i}{2}=n-1+\frac{\sum_{i=3}^{\infty}x_i}{2}.\]

Under our assumption that $\sum_{i=3}^{\infty}x_i>(n-1)/8$, we get $|\mathcal{F}|>17(n-1)/16$ and the proof is complete. 
\end{proof}

\begin{theorem}\label{threecolorspiders}
    If $T$ is a tree with $n$ vertices, then $c_3(T)\le 5(n-1)/4$. This bound is tight. 
\end{theorem}

\begin{proof}
Similarly to Theorem \ref{tree2}, we recursively construct a $3$-RSPS for the tree $T$. We initialize $T':=T$. At each step, given a bare spider of the remainder tree $T'$, we may remove from it any amount of $3$-stars or pairs of legs with a total of at least four edges, since any $m$ edges removed in this manner can be covered with at most $5m/4$ paths while preserving the bicolor edge-coverage property, by Theorem \ref{3colpath} and Theorem \ref{star3}. Thus, we may assume that the two longest legs of any bare spider of $T'$ have total length at most $3$ and that it has at most two legs of length $1$. 

Suppose that $T'$ is not a spider. This only leaves the cases $(2,1,1)$, $(2,1)$ and $(1,1)$, as a bare spider with one leg can again be deferred to a later step. The second and third cases have extended bare spiders $(2,1,1)$ and $(1,1,1)$, respectively. We show $3$-RSPS's for these two spiders that satisfy the bicolor edge-coverage property and have size at most $5/4$ times the number of edges in Figure~\ref{example} and in Figure~\ref{star_3col}, respectively.

\begin{figure}[h!]
    \centering
    \includegraphics[width=0.12\linewidth]{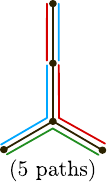}
    \caption{Optimal $3$-RSPS for the spider $(2,1,1)$.} 
    \label{example}
\end{figure}

On the other hand, if $T'$ is a spider, then it can only be $(2,1,1)$ (resolved as above), or a path with at most three edges, which has a $3$-RSPS of size at most the number of edges. 
\end{proof}

Finally, we show that four colors suffice for spiders and for pathy trees to attain the strong separation number.

\begin{theorem} \label{spider4}
    For every spider $S_{n,q}$ with $q\ge 3$, we have $c_k(S_{n,q})=n-1$ for every $k\geq 4$.
\end{theorem}

\begin{proof}

    The lower bound follows from the fact that $c_k(S_{n,q})\geq ssp(S_{n,q})=n-1$. We will prove the upper bound by an explicit construction.
    
    We begin by proving the statement for spiders $S_{n,q}$ with all legs of length $2$. We decompose the legs of $S_{n,q}$ into groups; specifically, we superimpose $\lfloor\frac{q}{3}\rfloor-1$ copies of the $4$-RSPS for $S_{7,3}$ given in Figure \ref{4spiders}. For the remaining legs --- either three, four or five depending on the value of $q \text{ mod } 3$ --- we apply the corresponding $4$-RSPS for $S_{7,3}$, $S_{9,4}$ or $S_{11,5}$, also shown in Figure \ref{4spiders}. 

    \begin{figure}[h!]
    \centering
    \includegraphics[width=1\linewidth]{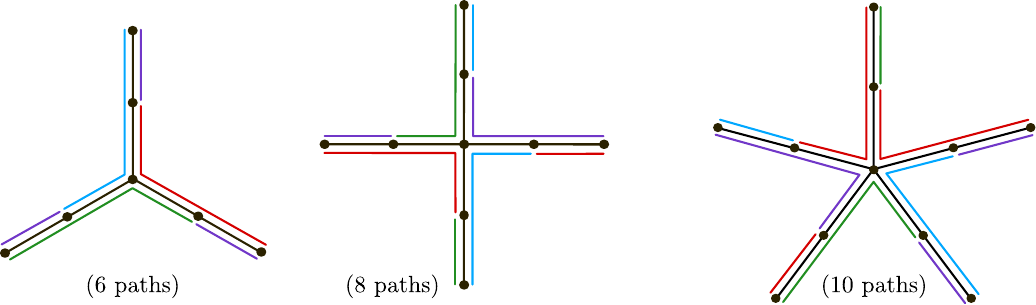}
    \caption{Optimal $4$-RSPS's for the spiders $S_{7,3}$, $S_{9,4}$ and $S_{11,5}$.} 
    \label{4spiders}
\end{figure}

    We then prove the upper bound for general spiders $S_{n,q}$ that are not paths. Let $S_{n,q}'$ be the spider obtained from $S_{n,q}$ by lengthening all of its legs of length $1$ by exactly one edge. We firstly construct a $4$-RSPS $\mathcal{F}'_n$ for $S_{n,q}'$. Let us denote by $S_{central}$ the \emph{central spider of length} $2$ in $S_{n,q}'$, i.e. the maximal spider that is a subgraph of $S_{n,q}'$ and has all legs of length $2$. We first consider the $4$-RSPS $\mathcal{F}_{central}$ of $S_{central}$ that is constructed as in the previous paragraph. Then, for each leg of $S_{n,q}'$ of length at least $3$, we extend the system in the same manner as for $3$-RSPS's of paths (see Figure \ref{path_3col}), using in each leg the colors of the three paths of $S_{central}$ that intersect it. Special attention is required in the case that the central spider has $2\pmod{3}$ legs; in order to do the 3-color path extension trick in the leg corresponding to the top leg of $S_{11,5}$ (see Figure \ref{4spiders}), we use blue as the third color. This yields a $4$-RSPS of $S_{n,q}'$ of size at most $|E(S_{n,q}')|$. Finally, we remove from $S_{n,q}'$ the edges that we added at the start. Each of these edges is contained in two paths of $\mathcal{F}'_n$, one of length $1$ and one of length $3$. We delete the former and prune the latter to obtain a $4$-RSPS $\mathcal{F}_n$ of $S_{n,q}$ that is of size $|E(S_{n,q})|=n-1$. 
\end{proof}

\begin{theorem}\label{fourcolortrees}
    If $T$ is a tree with $n$ vertices that is not a path, $c_k(T)\leq n-1$ for every $k\geq 4$. For pathy classes of trees, this bound is tight up to sublinear error.
\end{theorem}

\begin{proof}
    Similarly to Theorem \ref{tree2}, we construct our $4$-RSPS within a number of steps. In each step, we choose a bare spider, we find a $4$-RSPS for it (this time using the bound of Theorem \ref{spider4}), and we remove it. For a given bare spider $S$, if $S$ does not have $2\pmod{3}$ legs, we build for it a $4$-RSPS $\mathcal{F}_S$ with at most $e(S)$ paths and with the bicolor edge-coverage property (as in Theorem \ref{spider4}) and we update $T'$ by removing from it the edges of $S$. 

     If $S$ has $2\pmod{3}$ legs, then the construction of Figure \ref{4spiders} cannot be applied directly, as it does not have the bicolor edge-coverage property. Instead, we can opt to work with the extended bare spider $S^*$, which has $0\pmod{3}$ legs. We then build a $4$-RSPS for $S^*$ with at most $e(S^*)$ paths and with the bicolor edge-coverage property (as in Theorem \ref{spider4}) and we update $T'$ by removing from it the edges of $S^*$. We do this unless $T'$ is a spider and thus there is no hat edge, or when $T'$ consists of two spiders with $2\pmod{3}$ legs sharing a hat. If $T'$ is a non-path spider, we can use the construction from Theorem~\ref{spider4}. This will violate the bicolor edge-coverage property in one edge, which is acceptable for the last bare spider. The manner in which we handle the case of two spiders with $2\pmod{3}$ legs sharing a hat will ensure that whenever $T'$ is a spider, it is not a path.

     If $T'$ consists of two spiders with $2\pmod{3}$ legs sharing a hat, and the first spider has at least five legs, then we use Theorem \ref{spider4} to form a $4$-RSPS for the extended bare spider of the second spider and remove these edges from $T'$. Then we are left with a spider with at least five legs for which we may again use the construction from Theorem~\ref{spider4}.

     If $T'$ consists of two spiders with two legs each, sharing a hat, we use a construction based on extending or truncating the construction in Figure \ref{fig:placeholder}. The construction shown uses nine paths to separate a tree with nine edges. For any leg of length $1$, we can remove a terminal edge from the picture, thus eliminating one edge from the graph, as well as one path from the construction that is removed instead of shortened. For any leg of length more than $2$, we extend as before, extending a terminal single-edge path to contain the new edge, and adding a new single-edge path of the color used least recently on that leg. Thus, such a $T'$ with $m$ edges is separated via $m$ paths. 

     \begin{figure}
         \centering
         \includegraphics[width=0.85\linewidth]{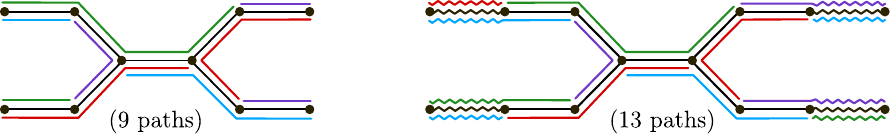}
         \caption{ An optimal $4$-RSPS for two paths attached with an edge.}
         \label{fig:placeholder}
     \end{figure}

    

    We let $\mathcal{F}$ be the union of all of the above $4$-RSPS's and note that its size is at most $n-1$. The lower bound for pathy classes follows directly from Proposition \ref{trees-strong}.
\end{proof}

\section{The sequence $r_k$ for classes of graphs}
\label{sec4}

In this section, we will study the family of parameters $r_k$, and especially the chromatic separation number $\chi_{rs}$ of various classes of graphs. We saw in Section~\ref{sec2} that the class $\mathcal{P}$ of paths has $\chi_{rs}(\mathcal{P})=2$. The same holds for the class $\mathcal{C}$ of cycles. For the class $\mathcal{S}$ of stars, we have $\chi_{rs}(\mathcal{S})=3$; for the class $\mathcal{SP}$ of spiders, $\chi_{rs}(\mathcal{SP})=4$; and $\chi_{rs}(\mathcal{PT})\leq 4$ is true for every pathy class of trees $\mathcal{PT}$. In contrast, we begin this section by proving that complete graphs are not rainbow separable.

\begin{theorem}\label{complete}
   For every $k\in\{2,3,\dots\}$, we have $r_k(\mathcal{K})\geq\frac{3}{2}$.
\end{theorem}

\begin{proof}

 For $n\in\mathbb{N}$, fixed $k$, and $\mathcal{F}_k$ a minimum $k$-RSPS of $K_n$, let $E_{\leq 2}$ be the set of edges of $K_n$ belonging to $\leq 2$ paths in $\mathcal{F}_k$. Let $G_{\mathcal{F}_k}$ be an auxiliary graph, each vertex of which corresponds to a path in $\mathcal{F}_k$ and each edge of which corresponds to an intersection of two such paths along an edge in $E_{\leq 2}$. By Proposition \ref{cliques} and Proposition \ref{inequality}, this is a simple graph of size $\leq 2n+10$. Moreover, the neighborhood of every vertex of $G_{\mathcal{F}_k}$ corresponds to paths of distinct colors, as the situation in Figure~\ref{contradiction} cannot be encountered in $\mathcal{F}_k$. That is, the degree of each vertex in $G_{\mathcal{F}_k}$ is bounded above by $k$. We conclude that $|E_{\leq 2}|=O_k(n)$, which implies that $|\mathcal{F}_k|\geq \frac {3({n \choose 2}-O_k(n))}{n-1}\simeq \frac{3}{2}n$, so $r_k(\mathcal{K})\geq\frac{3}{2}$ for every $k$. 

 \end{proof}

\begin{figure}[h!]
	\centering
	\includegraphics[width=0.5\textwidth]{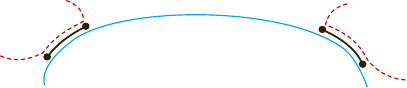}
	\caption{The two marked edges in $E_{\leq 2}$ are not separated, so this cannot occur in $\mathcal{F}_k$.}
    \label{contradiction}
\end{figure}

In fact, we can say something much more general. In \cite{nico}, Cristina Fernandes, Guilherme Oliveira-Mota and Nicolás Sanhueza-Matamala proved that every graph that is almost $\alpha n$-regular and robustly connected (see \cite{nico} for precise definitions) has strong separation number $(\sqrt{3\alpha+1}-1+~o(1))n$. On the other hand, the argument from Theorem~\ref{complete} yields a lower bound of $\frac{3}{2}\alpha$ for every rainbow separation ratio of any class of such graphs, hence:

\begin{theorem}\label{densetheorem}
    If $\Sigma$ is a class of dense, almost-regular, robustly connected graphs, then it is not rainbow separable.  
\end{theorem}

Such classes are, for example, the complete graphs and the Erd\H{o}s-R\'enyi random graphs $G(n,p)$ with constant $p$. 

Starting from Theorem \ref{complete}, we can also easily obtain classes that have a given rainbow separation ratio in $[1,\frac{3}{2}]$.

\begin{corollary}
    For every $k\in\{2,3,\dots\}$ and $r\in [1,r_k(\mathcal{K})]$, there is a class $\Sigma$ of connected graphs with $r_k(\Sigma)=r$.
\end{corollary}

\begin{proof}
    Let $\lambda:=\frac{r-1}{r_k(\mathcal{K})-1}$ and let $\Sigma$ be the class of graphs $G_n$ that consist of a path $P^{\lambda}(n):=P_{\lceil(1-\lambda)n\rceil}$ ending at a complete graph $K^{\lambda}(n):=K_{\lfloor\lambda n\rfloor}$. By treating $P^{\lambda}(n)$ and $K^{\lambda}(n)$ separately we can show that $c_k(G_n)\leq (1-\lambda)n+2+(r_k(\mathcal{K})+o(1))(\lambda n+9)=rn+o(n)$ and $c_{\infty}(G_n)\leq n+8$. On the other hand, let $\mathcal{F}_n$ be a strongly separating path system of $G_n$. For every edge of $P^{\lambda}(n)$ except for the last, there must be a path in $\mathcal{F}_n$ that separates this edge from the one that is incident to it and closer to $K^{\lambda}(n)$. This yields at least $(1-\lambda)n-2$ paths that do not intersect $K^{\lambda}(n)$. Moreover, any nonempty edge-intersection of a path in $\mathcal{F}_n$ with $K^{\lambda}(n)$ is a path, and at least $\lambda n-2$ paths are required to strongly separate the edges of $K^{\lambda}(n)$, so $\mathcal{F}_n$ contains at least $\lambda n-2$ paths that do intersect $K^{\lambda}(n)$. Therefore, $c_{\infty}(G_n)\geq n-4$, so $\frac{c_k(G_n)}{c_{\infty}(G_n)}\le r+o(1)$. The previous lower bound argument can be replicated to show that for large $n$ we have $c_k(G_n)\geq (1-\lambda)n-2+(r_k(\mathcal{K})-o(1))(\lambda n-2)=rn-o(n)$. Thus, $\frac{c_k(G_n)}{c_{\infty}(G_n)}=r+o(1)$, so $r_k(\Sigma)=r$.       
\end{proof}

In general, if for the class of all connected graphs $\mathcal{G}$ we have $r_k(\mathcal{G})=c$, then we can adjust the proof of the above theorem to show that every number in $[1,c]$ is the $k$-rainbow separation ratio of a class of connected graphs. The only meaningful question about the range of $r_k$ is therefore the following:

\begin{question}
    What is $r_k(\mathcal{G})$ for each $k$?
\end{question}

On a different note, one might also wonder about whether there are classes that are rainbow separable but have $\chi_{rs}=\infty$. Below is an example. We say that a binary tree is \emph{full} if every vertex apart from the leaves and the root has degree $3$. We say that a full binary tree is \emph{complete} if every leaf has the same distance from the root.

\begin{theorem}\label{binary}
    Let $\mathcal{BT}$ be the class of complete binary trees. Then $\chi_{rs}(\mathcal{BT)}=\infty$.
\end{theorem}

\begin{proof}
    Let $\mathcal{F}_n$ be a $k$-RSPS of the complete binary tree on $n$ vertices, which we denote $T_{n,2}$, where $n$ is of the form $2^r-1$. Suppose that $|\mathcal{F}_n|=\frac{n+1}{2}+\varepsilon (n+1)$ for $\varepsilon> 0$. At most $k$ edges of $T_{n,2}$ may each be in a unique path of length $1$ of $\mathcal{F}_n$. Every other edge, including those incident to the leaves, must be in at least two paths. Let $X$ be the set of leaf-edges that are only contained in paths that are bounded by leaves, let $Y$ be the set of leaf-edges that are only contained in paths of length at least $2$, at least one of which is not bounded by leaves, let $Z$ be the set of leaf-edges that are contained in one path of length $1$ and are otherwise contained only in paths bounded by leaves, and let $W$ be the set of leaf-edges that are contained in at least two paths that are not bounded by leaves, including at least one path of length $1$. Note that $X$, $Y$, $Z$ and $W$ partition the leaf-edges of $T_{n,2}$, i.e. $|X|+|Y|+|Z|+|W|=\frac{n+1}{2}$. Each edge in $X$ must be in at least two paths. Each edge in $Y$ must be in at least one path bounded by leaves, in order to be distinguished from its parental edge. So, in total, the leaf-edges of $T_{n,2}$ account for at least $|X|+\frac{|Y|}{2}+\frac{|Z|-k}{2}$ paths bounded by leaves. Moreover, the edges in $Z$ and $W$ contribute $|Z|+|W|$ more paths of length $1$. Finally, the edges in $Y$ and $W$ account for at least $|Y|+|W|$ more paths that are not bounded by leaves. Thus, \[|X|+\frac{3}{2}|Y|+\frac{3|Z|-k}{2}+2|W|\leq \frac{n+1}{2}+\varepsilon (n+1)\Rightarrow\] \[2|X|+3\left(\frac{n+1}{2}-|X|-|Z|-|W|\right)+3|Z|-k+4|W|\leq n+1+2\varepsilon (n+1)\Rightarrow\] \[|X|\geq\frac{n+1}{2}-2\varepsilon (n+1)-k+|W|.\] So, at least $\frac{n+1}{2}-2\varepsilon (n+1)-k$ leaf-edges of $T_{n,2}$ must only be in paths of $\mathcal{F}_n$ that are bounded by two leaves. In particular, there are at least $\frac{n+1}{2}-2\varepsilon (n+1)-k$ such paths in $\mathcal{F}_n$, so there are at most $3\varepsilon (n+1)+k$ paths in $\mathcal{F}_n$ that are not bounded by leaves. Since each path contains at most two edges in each level, all but at most $6\varepsilon (n+1)+2k$ edges in each level are each contained only in paths bounded by leaves.  

    We also note that for each internal vertex of $T_{n,2}$, there exists at least one path in $\mathcal{F}_n$ that contains this vertex but not its parent. This means that we have $\frac{n+1}{4}$ paths in $\mathcal{F}_n$ of length at most $2$, $\frac{n+1}{8}$ additional paths of length at most $4$, $\frac{n+1}{16}$ additional paths of length at most $6$, $\dots$, two of length at most $2(r-2)$. Moreover, there are two paths that separate the two edges incident to the root. In total, $\frac{n+1}{2}$ of the paths in $\mathcal{F}_n$ have a total length of at most $\sum_{i=1}^{\infty}\frac{i}{2^i}(n+1)+2(r-1)=2(n+r)$. It is easy to see that the remaining $\varepsilon (n+1)$ paths cover at most $(3+2\log_2{\varepsilon}^{-1})\varepsilon (n+1)$ distinct edges. Indeed, they may cover the at most $\varepsilon (n+1)$ edges of the initial subtree of $T_{n,2}$ on $\floor{\log_2 \varepsilon (n+1)}$ levels, as well as at most $2\varepsilon (n+1)$ edges on each remaining level, where the number of the remaining levels is $\log_2 (n+1)-\floor{\log_2{\varepsilon (n+1)}}\leq \log_2{\varepsilon^{-1}}+1$. There are at least $n-k-1$ edges that are covered at least twice by paths of $\mathcal{F}_n$, and at least $n-k-1-((3+2\log_2{\varepsilon}^{-1})\varepsilon (n+1))$ of them are only covered by the $\frac{n+1}{2}$ paths of total length at most $2(n+r)$. Thus, at most $2(n+r)-2(n-k-1-((3+2\log_2{\varepsilon}^{-1})\varepsilon (n+1)))=2(r+k+1)+2((3+2\log_2{\varepsilon}^{-1})\varepsilon (n+1))$ such edges can be covered more than twice. Along with the up to $(3+2\log_2{\varepsilon}^{-1})\varepsilon (n+1)$ edges covered by the remaining $\varepsilon(n+1)$ paths and at most $k$ edges covered by only one path, there are at most  $2(r+k+1)+(9+6\log_2{\varepsilon}^{-1})\varepsilon (n+1)+k$ edges not covered by exactly two paths.

   From all of the above, and as $6\varepsilon (n+1)+2k+2(r+k+1)+(9+6\log_2{\varepsilon}^{-1})\varepsilon (n+1)+k<16\sqrt{\varepsilon}n$ for $n$ large enough, we deduce that all but at most $16\sqrt{\varepsilon}n$ edges in each level of $T_{n,2}$ lie in exactly two paths of $\mathcal{F}_n$, both of which are bounded by leaves. In particular, when $h:=\floor{\frac{1}{4}\log_2{\varepsilon^{-1}}}-5$ is positive, there is a final subtree $T$ of height $h$ that contains only such edges. This is because $16h\sqrt{\varepsilon}n<\frac{n+1}{2^{h+1}}$, that is, there are more final subtrees of height $h$ than there are ``bad" edges up to height $h$. There exists a path $P$ in $\mathcal{F}_n$ such that the subpath $P\cap T$ has length at least $h$ and contains the root of $T$. All the edges in $P\cap T$ are contained in exactly one more path, and those other paths are pairwise distinct and have distinct colors, as otherwise the edges of $P\cap T$ would not all be separated from each other. Thus, there are at least $h$ colors. We conclude that $h\leq k$, so $\varepsilon\geq \frac{1}{16^{k+6}}>0$. All of the above was done under the assumption that $\varepsilon>0$; if $\varepsilon=0$, then we can add a linear number of junk paths to $\mathcal{F}_n$ (but fewer than $\frac{n+1}{16^{k+6}}$) and obtain a contradiction. 

    \begin{figure}
    \centering
    \includegraphics[width=0.4\linewidth]{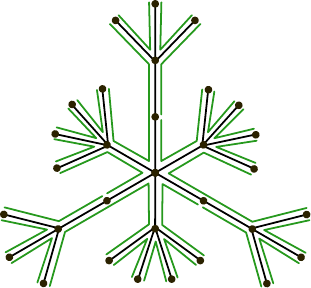}
    \caption{The construction of Balogh et al. for the strong separation of trees.}
    \label{fig:balogh}
\end{figure}

    For the opposite inequality, for $n$ large enough, we take the strongly separating path system $\mathcal{F}_n$ constructed in the proof of Theorem 5 of Balogh et al. \cite{balogh2016path} (see Figure \ref{fig:balogh}), and we add to it as a path each edge that is not in the last $\floor{\log_2(2k)}-1$ edge-levels. We visit each of the subtrees of $T_{n,2}$ in the last $\floor{\log_2 (2k)}$ levels and color each of the at most $k-1$ paths that live in it a different color among the first $k-1$ colors in our possession. Each subtree in the last $\floor{\log_2(2k)}$ levels has $2(\floor{\log_2(2k)}-1)$ edges contained in just one path staying entirely within that subtree, and we also add each such edge as a path to our system. We then color every remaining path from the original construction arbitrarily with one of the first $k-1$ colors, and we color both types of added single-edge paths with the $k^{th}$ color. 
    
    It is easy to see that this is a $k$-RSPS of $T_{n,2}$. Indeed, for two edges $e,f$, either $e$ is contained in a single-edge path of color $k$ and $f$ is contained in a path from the original construction which does not include $e$ and is some other color, or both $e,f$ are contained in two paths entirely contained in a subtree in the last $\floor{\log_2(2k)}$ levels and these paths have two different colors from the first $k-1$. In that case, $e,f$ are still separated from each other as the respective paths containing each are contained in separate subtrees within the last $\floor{\log_2(2k)}$ levels and we can simply pick any path containing $e$ and one of a different color containing $f$, or $e,f$ are in the same subtree within the last $\floor{\log_2(2k)}$ levels and the paths used to separate them in the original construction were assigned different colors. This $k$-RSPS uses at most \[\frac{n+1}{2}+2^{\log_2 (n+1)-(\floor{\log_2 (2k)}-1)}-2+(2^{\log_2(n+1)-\floor{\log_2(2k)}})(2(\floor{\log_2(2k)}-1))\leq\] \[\left(\frac{n+1}{k}\right)\frac{k+2\log_2k+6}{2}\]
    paths. When dividing by $c_{\infty}(T_{n,2})=\frac{n+1}{2}$, this tends to $1$ as $k\to\infty$.
\end{proof}

\begin{remark}\label{extension}
    The proof of Theorem \ref{binary} can easily be adapted to some structured perturbations of complete binary trees, for example to full binary trees in which only the last level is not complete. Indeed, the above proof only uses that the number of leaves is the ceiling of half the number of vertices (which is true for every full binary tree), and that, for every fixed height $h$, the number of terminal complete binary subtrees of $T$ of height $h$ grows linearly with the number of leaves. Of course, the constant $\varepsilon$ will be different.        
\end{remark}

Theorem \ref{binary} has an interesting consequence. Given a graph $G$, we define the \emph{chromatic separation number} $\chi_{rs}(G)$ ass the minimum $k$ such that $r_k(G)=1$, that is $c_k(G)=c_{\infty}(G)$. Constructing graphs with specified chromatic numbers may be useful in constructing classes of graphs with specified chromatic numbers.

\begin{corollary}\label{corollary}
    For $k\geq 3$, there are infinitely many trees with chromatic separation number $k$.
\end{corollary}

\begin{proof}
We have already encountered infinitely many trees with chromatic separation numbers $3$ and $4$, namely stars with at least five legs and spiders with at least twenty-five legs all of which have length $2$, respectively. For the complete binary tree of depth $3$, we have that $c_5(T_{7,2})=c_{\infty}(T_{7,2})=5$. So, $\chi_{rs}(T_{7,2})\leq 5$. 

Let $T'$ be a tree and suppose that $\chi_{rs}(T')=k-1$, as witnessed by a $(k-1)$-RSPS $\mathcal{F'}$. Suppose, moreover, that there exists a leaf $v$ of $T'$ such that the edge incident to it is contained in at least two paths of $\mathcal{F'}$ of different colors, say $P_x$ and $P_y$. Let us attach to $v$ two new vertices $x$ and $y$. We extend $P_x$ to $x$ and $P_y$ to $y$, and then we add to $\mathcal{F'}$ a new path $xvy$, colored in a new color. We end up with a new tree $T$ and a $k$-RSPS $\mathcal{F}$ such that $|\mathcal{F}|=|\mathcal{F}'|+1$. As $T$ has exactly one more leaf than $T'$ and the same number of degree $2$ vertices, by Proposition \ref{trees-strong}, we get $c_{\infty}(T)=c_{\infty}(T')+1=|\mathcal{F'}|+1=|\mathcal{F}|\geq c_k(T)\Rightarrow c_k(T)=c_{\infty}(T)$. So, $\chi_{rs}(T)\leq k$. Moreover, this operation does not decrease $\chi_{rs}$. Indeed, if there exists $k'\leq k-2$ such that $c_{k-2}(T)=c_{\infty}(T)$ as witnessed by a $(k-2)$-RSPS $\mathcal{H}$, then $\mathcal{H}$ induces a $(k-2)$-RSPS $\mathcal{H}'$ in $T'$ such that $|\mathcal{H}'|\leq |\mathcal{H}|-1$, because $\mathcal{H}$ contains a path that separates $\{v,x\}$ from its parental edge, so this path is absent in $\mathcal{H}'$ (rather than simply pruned). On the other hand, $|\mathcal{H}'|\geq c_{\infty}(T')=c_{\infty}(T)-1=|\mathcal{H}|-1$, so $|\mathcal{H}'|=|\mathcal{H}|-1=c_{\infty}(T')$, implying that $\chi_{rs}(T')\leq k-2$, which contradicts our original assumption. So, $k-1\leq \chi_{rs}(T)\leq k$. 

Now, let $t\geq k\geq 5$, and let us consider the tree $T^0(t)$ that is formed by taking $t$ copies of $T_{7,2}$, choosing in each one a child of its root, and merging the selected children to a vertex $u$. In particular, whereas $T_{7,2}$ has four leaves and one vertex of degree $2$, $T^0(t)$ has $4t$ leaves and $t$ vertices of degree $2$. We have $c_5(T^0(t))\geq c_{\infty}(T^0(t))=5t=t\cdot c_{\infty}(T_{7,2})=t\cdot c_5(T_{7,2})\geq c_5(T^0(t))$, so $\chi_{rs}(T^0(t))\leq 5$. Let $k_0:=\chi_{rs}(T^0(t))$, as witnessed by a minimal $k_0$-RSPS $\mathcal{F}_0$. There are at most $k_0\leq k$ edges in $T^0(t)$ that are contained only in a monochromatic set of paths in $\mathcal{F}_0$; in particular, if $k_0<k\leq t$, then one of the $t$ original copies of $T_{7,2}$ in $T^0(t)$ does not contain any of them. So, we can perform the operation of attaching two new leaves on any leaf of that copy, as described in the previous paragraph. We denote the resulting tree by $T^1(t)$, and we let $k_1:=\chi_{rs}(T^1(t))$, and $\mathcal{F}_1$ be a minimal $k_1$-RSPS of $T^1(t)$. Then, if $k_1<k\leq t$, again we can perform our operation on a leaf of one of the $t$ canonical subtrees of $T^1(t)$ that does not contain monochromatic edges with respect to $\mathcal{F}_1$, thus obtaining $T^2(t)$, and so on. Each time, we perform the operation in such a way that every canonical subtree ends up having at most two terminal levels. 

As long as $k_r<k$, we can continue repeating the operation. Because of how we complete levels, as $r$ increases, increasingly larger members of $\mathcal{BT}$ appear as subtrees of some canonical subtrees $S^{r,1},\dots,S^{r,s}$ of $T^r(t)$, where $1\leq s\leq t$. Let $r$ be large enough, say $r>>\varepsilon^{-1}$, where $\varepsilon$ is the constant obtained from Theorem \ref{binary} and Remark \ref{extension} for full binary trees with all levels complete except perhaps for the last one, and for $k_r$ colors. Suppose that $S^{r,1},\dots,S^{r,s}$ have $\ell_{r,1},\dots,\ell_{r,s}$ leaves, respectively. Then for each $i\in [s]$, $c_{k_r}(S^{r,i})\geq (1+\varepsilon)c_{\infty}(S^{r,i})$. Moreover, as $r$ is large enough, these unbounded canonical subtrees together contain at least $(1-\varepsilon^2)\ell_r$ leaves, where $\ell_r:=4t+r$ is the number of leaves of $T^r(t)$. If $\mathcal{F}_r$ is a $k_r$-RSPS of $T^r(t)$ with $c_{\infty}(T^r(t))$ paths, then it induces, in each $S^{r,i}$, a path system of at least $(1+\varepsilon)(\ell_{r,i}+1)$ paths. Out of these, $\ell_{r,i}-3$ do not contain the vertex $u$, as they are required to separate, for each of the vertices of degree $3$ (except $u$) of $S^{r,i}$, one of the edges below it from the edge above it. Each of the other paths in $\mathcal{F}_r$ may be shared between two canonical subtrees of $T^r(t)$. In total,

\[|\mathcal{F}_r|\geq \left((1-\varepsilon^2)\ell_r-3s\right)+\frac{(1+\varepsilon)\left((1-\varepsilon^2)\ell_r+s\right)-\left((1-\varepsilon^2)\ell_r-3s\right)}{2}+\frac{\varepsilon^2\ell_r+(t-s)}{2}=\] \[(1+\frac{\varepsilon}{2}-\frac{\varepsilon^2}{2}-\frac{\varepsilon^3}{2})\ell_r+-\frac{3-\varepsilon}{2}s+\frac{t}{2}>\ell_r+t=c_{\infty}(T^r(t)),\] a contradiction.   

So, we cannot perform the operation indefinitely; there exists some minimum $r_0\in\mathbb{N}$ such that $k_{r_0}\geq k$. As the increment of $k_r$ is at most $1$, we have $k_{r_0}=k$. As $t\geq k$ is arbitrary, this yields infinitely many trees with chromatic separation number $k$.

\end{proof} 

\section{Future Work}

We wonder whether the construction of Corollary \ref{corollary} can be modified to yield a class of connected graphs with chromatic separation number $k$. Note that, if one has infinitely many graphs of chromatic separation number $k$, the family formed by these graphs could have chromatic separation number strictly less than $k$. To avoid this, we additionally require that any $(k-1)$-RSPS on sufficiently large members of the class to have a linear surplus of paths with respect to $c_{\infty}$. So, lumping together our infinitely many examples obtained from Corollary \ref{corollary} is not sufficient. 

One could attempt to obtain an example by gluing copies of our tree $T^{r_0}(t)$ along a vertex, for some fixed $t\geq k$. In this case, we suspect that the number of paths in excess of $c_{\infty}$ scales linearly with the number of copies, as required. We did not manage to prove this, because some paths can contain edges from two copies, which significantly affects the calculation. We therefore pose the following question. 

\begin{question}
    Does there exist, for every $k\geq 2$, a class $\mathcal{C}$ of graphs with $\chi_{rs}(\mathcal{C})=k$?
\end{question}

Let us also reiterate the question that lies at the heart of this endeavor: what is it that controls the behavior of the sequence $r_k$ for a class of graphs $\Sigma$? In particular, what determines the chromatic separation number? In a companion note, Micha Christoph and the second author present an example of a class of dense graphs that is rainbow separable. They also prove that the class of all trees is rainbow separable; in particular, this is a rainbow separable class of unbounded pathwidth. In light of this last result, we propose the following conjecture.   

\begin{conjecture}
Any class of graphs of bounded treewidth is rainbow separable.
\end{conjecture}

\bibliographystyle{plain}
\bibliography{main}

\appendix

\section{On the lower bound of Theorem \ref{thm:spiders}}\label{lowerbound}
The lower bound in Theorem~\ref{thm:spiders} is not always tight. Because $c_2(S_{n,q})\ge c_2(S_{n-1,q})+1$, the lower bound is only tight if every edge added in the process of extending $S_{2q+1,q}$ to $S_{n,q}$ increases $c_2$ by exactly $1$ and the restriction of the RSPS to the inner $S_{2q+1,q}$ has $2q+\lfloor{2q/3\rfloor}$ paths if $q\equiv 1 \text{ mod } 3$ or $q=3$, and $2q-1+\lfloor{2q/3\rfloor}$ paths otherwise. 

We now consider the implication for spiders with five or six legs. To have any chance of obtaining the lower bound $n-2+\lfloor{\frac{2q}{3}\rfloor}=n+1$ for a spider with five legs, the inner $S_{11,5}$ must have a $2$-RSPS of size $12$. First we show that in such a $2$-RSPS, each terminal edge of the inner $S_{11,5}$ is contained in exactly one path of length $1$. Assume for the sake of contradiction that some such edge is contained in two such paths, which are clearly separate colors. Then, the remaining four terminal edges need to be contained in paths of length one. Without loss of generality, either
\begin{itemize}
    \item[(i)] All four remaining terminal edges are contained in red paths of length $1$,
    \item[(ii)] Three of the remaining terminal edges are contained in red paths of length $1$ and the other is contained in a blue path of length $1$, or
    \item[(iii)] Two of the remaining terminal edges are contained in red paths of length $1$ and the other two are contained in blue paths of length $1$.
\end{itemize}

In case (i), the remaining red paths form a weakly separating path system for a star with five edges, so there are at least three by Proposition \ref{trees-weak}. The remaining blue paths form a weakly separating path system for a spider with four legs of length $2$ and one of length $1$, so by Proposition \ref{trees-weak}, there are at least $\lceil{\frac{2(5-1)+4}{3}\rceil}=4$. This gives a total of $6+3+4=13$ paths, which is too many.

In case (ii), the remaining red paths form a weakly separating path system for a spider with four legs of length $1$ and one of length $2$, and the blue paths form a weakly separating path system for a spider with three legs of length $2$ and two of length $1$. Thus, by Proposition \ref{trees-weak}, there are at least $\lceil{\frac{2(5-1)+1}{3}\rceil}=3$ additional red and $\lceil{\frac{2(5-1)+3}{3}\rceil}=4$ additional blue paths. Again, this gives a total of $6+3+4=13$ paths, which is too many.

In case (iii), the remaining red paths form a weakly separating path system for a spider with three legs of length $1$ and two of length $2$. The remaining blue paths do as well, so by Proposition~\ref{trees-weak}, there are a total of at least $6+2\lceil{\frac{2(5-1)+2}{3}\rceil}=14$ paths, a contradiction.

Thus, when restricting a rainbow path separating system of size $n+1$ for $S_{n,5}$ to the inner $S_{11,5}$, each terminal edge of the inner $S_{11,5}$ must be contained in exactly one path of length $1$.

To have any chance of obtaining the lower bound $n-2+\lfloor{\frac{2q}{3}\rfloor}=n+2$ for six legs, the inner $S_{13,6}$ must have a $2$-RSPS of size $15$. First we show that in such a $2$-RSPS, each terminal edge of the inner $S_{13,6}$ is contained in exactly one path of length $1$. Assume for the sake of contradiction that some such edge is contained in two such paths, which are clearly separate colors. Then the remaining five terminal edges need to be contained in paths of length $1$. Without loss of generality, either
\begin{itemize}
    \item [(i)] All five remaining terminal edges are contained in red paths of length $1$,
    \item[(ii)] Four of the remaining terminal edges are contained in red paths of length $1$ and the other is contained in a blue path of length $1$, or
    \item[(iii)] Three of the remaining terminal edges are contained in red paths of length $1$ and the other two are contained in blue paths of length $1$.
\end{itemize}
The remaining red paths must form a weakly separating path system for the edges not yet covered by red paths and the remaining blue paths must form a weakly separating path system for the edges not yet covered by blue paths. By Proposition \ref{trees-weak}, this gives lower bounds of $\lceil{\frac{2(6-1)+5}{3}\rceil}+\lceil{\frac{2(6-1)}{3}\rceil}, \lceil{\frac{2(6-1)+4}{3}\rceil}+\lceil{\frac{2(6-1)+1}{3}\rceil}$, and $\lceil{\frac{2(6-1)+3}{3}\rceil}+\lceil{\frac{2(6-1)+2}{3}\rceil}$, respectively for the number of additional paths in the three cases, all of which are equal to $9$. Together with the at least seven paths of length $1$ for the terminal edges, this is at least $16$, a contradiction.

For $q=5,6$, suppose that some edge $e$ in the inner star $S_{q+1}$ is contained in exactly one red path and one blue path. The next edge, $f$, along the same leg is contained in exactly one path that does not contain $e$. Without loss of generality, this path is red. Then, either $f$ is contained in no blue path, or it is contained only in the blue path containing $e$. In the latter case, there is no way to separate $e,f$ since there is no blue path that contains exactly one of them. Thus $f$ is a monochromatic edge contained in one path. 

Therefore, either every edge in the inner $S_{q+1}$ is contained in at least three paths, or the inner $S_{2q+1,q}$ has some monochromatic edge contained in at most two paths. The first scenario cannot happen for $q=5$. Indeed, as each terminal edge of the $S_{11,5}$ is contained in a path that misses the penultimate edge on the same leg, there are at most $12-5=7$ paths that have a nonempty intersection with the inner $S_6$. Each of these contains at most two edges of the inner $S_6$ for a total of just $14$ path--edge incidences, not enough for each of the five edges of the $S_6$ to be contained in at least three paths. Thus, any possible $2$-RSPS of the desired size is such that some edge of the inner $S_{11,5}$ is monochromatic and contained in at most two paths. For $S_{n,6}$, there is also the option where every edge in the inner $S_{7}$ is contained in at least three paths.

For the eventual ease of handling this latter case for $S_{n,6}$, we will consider what happens when some edge of the inner $S_{3q+1,q}$, that is within the first three edges along each leg, is a monochromatic edge contained in at most two paths. For $q=5,6$, for a spider $S_{n,q}$ where all legs are arbitrarily long and, with respect to the $2$-RSPS, some edge of the inner $S_{3q+1,q}$ is monochromatic and contained in at most two paths, it means that even a leg of the inner $S_{3q+1,q}$ that includes a monochromatic edge contained in at most two paths was extended arbitrarily many times, only increasing $c_2$ by $1$ each time. We will show that this cannot happen.

Without loss of generality, assume that some edge in the inner $S_{3q+1,q}$ is contained in no red paths and at most two blue paths. Call this edge $e_1$, and let $e_2,e_3,\dots,e_l$ be the subsequent edges (in order) along this leg with $e_l$ incident to a leaf of $S_{n,q}$ and $l\ge 5$. Note that for $j=2,\dots,l$, there is exactly one path containing $e_j$ but $e_{j-1}$. If $e_j$ is not part of the inner $S_{2q+1,q}$, this is because the lower bound $n-2+\lfloor{2q/3\rfloor}$ can only be obtained if every edge added in the process of extending $S_{2q+1,q}$ to $S_{n,q}$ increases $c_2$ by exactly $1$. If $e_1$ is in the inner star and thus $e_2$ is part of the inner $S_{2q+1,q}$, it is because we have already established that for a $2$-RSPS of size $12$ for $S_{11,5}$ or of size $15$ for $S_{13,6}$, each terminal edge is contained in only one path of length $1$. 

For $j\ge 2$, there must be a pair of paths separating $e_j$ and $e_1$ and the path containing $e_j$ is necessarily red. Any red path containing $e_2$ does not contain $e_1$ but there is only one path containing $e_2$ but not $e_1$ so there is only one red path containing $e_2$. In order to separate $e_1,e_2$, at least one of the blue paths containing $e_1$ does not contain $e_2$, so $e_2$ is contained in at most one blue path. There is one path containing $e_3$ but not $e_2$. If it is blue, then for $e_3$ to be contained in a red path, it must be a red path that contains $e_2$. Since there is only one red path containing $e_2$, $e_2$ and $e_3$ would be contained in the same red paths and could not be separated. Thus the path containing $e_3$ but not $e_2$ is red. To separate $e_2,e_3$, there must be a blue path containing $e_2$ but not $e_3$. Since there is at most one blue path containing $e_2$, there must be exactly one blue path containing $e_2$ and none containing $e_3$. Thus $e_3$ is only contained in red paths, so to separate $e_3,e_4$, there must be a blue path containing $e_4$ but not $e_3$. Additionally $e_4$ must be contained in a red path, so it must be contained in a red path which also contains $e_3$. To separate $e_3,e_4$, they cannot be contained in the exact same red paths, so there must be an additional red path that contains $e_3$. That is, $e_3$ is contained in the single red path containing $e_2$, meaning $e_3$ is contained in exactly two red paths. If $e_4$ is contained in both, there is no path with $e_3$ but not $e_4$. Thus $e_4$ is contained in exactly one of the red paths containing $e_3$. If it is contained in the red path which also contains $e_2$, then $e_2,e_4$ are contained in the exact same red paths and cannot be separated. Thus $e_4$ is necessarily contained in the red path that has $e_3$ but not $e_2$.
Note that $e_5$ must be contained in both a red path and a blue path to be separated from $e_1$ and $e_3$, respectively. It is only contained in one path that does not contain $e_4$, so it must also be contained in a path containing $e_4$. However, $e_4$ is in only one path of each color, so either way, there is some color where $e_4,e_5$ are only contained in one path of that color and it is the same path, so they cannot be separated. Thus, there exists some spider (and hence infinitely many) with five legs, all of length at least two such that the lower bound of $n-2+\lfloor{\frac{2q}{3}\rfloor}$ is not tight. Additionally, this eliminates one of the two remaining cases for six legs. That is, for a $2$-RSPS for $S_{n,6}$ with arbitrarily long legs to obtain the lower bound of $n+2$, each edge of the inner $S_7$ must be contained in at least three paths. We may now assume that any monochromatic edge in the inner $S_{19,6}$ is also contained in at least three paths.

No path can contain more than two edges of the inner $S_7$. Furthermore, since only fifteen paths contain edges of the inner $S_{13,6}$ and six of those only contain terminal edges of the $S_{13,6}$, there are nine paths that contain any edges of the inner $S_7$. Therefore, each of these paths must contain exactly two edges of the inner $S_7$ and each edge of the inner $S_{7}$ is contained in exactly three paths. Among these nine paths, at least four are red and at least four are blue since $wsp(S_7)=4$. Without loss of generality, assume exactly four are blue. As at most one edge is contained only in blue paths and at most one is contained only in red paths, only the following distributions are possible for the number of blue paths containing each of the six edges of the inner $S_7$: (i) $3,1,1,1,1,1$, (ii) $2,2,1,1,1,1$, (iii) $3,2,1,1,1,0$, (iv) $2,2,2,1,1,0$.

Firstly, observe that case (i) is impossible. Indeed, each edge contained in a single blue path must be in a separate blue path, which means there are at least five blue paths on these edges instead of four.

In case (ii), let $a,b$ denote the two edges of the inner $S_7$ contained in two blue paths. Since each edge contained in a single blue path must be in a separate blue path, then each of these four blue paths must contain exactly one of $a,b$ and exactly one of the other four. Consider the terminal edges of $S_{13,6}$ incident to $a,b$. If one is monochromatic, it must be contained at least three paths. It is contained in one path of length $1$ (when restricted to $S_{13,6}$) and at most one other red path since the adjacent inner edge is contained in just one red path. Thus, such a terminal edge can only be monochromatic if it is contained in just blue paths. At most one edge is contained in just blue paths, so without loss of generality, the terminal edge $a'$ incident to $a$ is not monochromatic. Edge $a'$ must be contained in a red path of length $1$ (when restricted to $S_{13,6}$) as otherwise, the only red path containing it would be the single red path containing $a$ and there would be no way to separate $a,a'$. Thus, the blue paths containing $a'$ are a subset of those containing $a$. In fact, $a'$ is contained in exactly one of those blue paths: if contained in none, $a'$ is monochromatic, but if contained in both, then $a,a'$ are contained in the exact same blue paths and cannot be separated. Note that besides containing $a,a'$, this blue path also contains some other edge of the inner $S_7$ which was only contained in one blue path. That edge and $a'$ are contained in the exact same blue paths, and cannot be separated. Thus, case (ii) is also impossible.

For cases (iii) and (iv), we will establish that there is a monochromatic edge in the inner $S_7$ such that any path containing it also contains an edge from a different leg of the inner $S_{13,6}$ contained in no other paths of that color. We will show that similarly to when there is a monochromatic edge contained in just two paths, the leg containing such a monochromatic edge cannot be extended arbitrarily many times while only increasing $c_2$ by $1$ each time.

For case (iii), any terminal edge of the inner $S_{13,6}$ is contained in paths of both colors since there is already a monochromatic red edge and a monochromatic blue edge in the inner $S_7$. A terminal edge incident to an inner edge contained in at most one red (blue) path must be covered by a red (blue) path of length $1$ (when restricted to $S_{13,6}$) since otherwise the terminal edge is contained in no red (blue) paths or precisely the same set of red (blue) paths as the inner edge. Let $a$ be the inner edge covered by exactly two blue paths and $a'$ be its incident terminal edge. Since $a'$ is not monochromatic, it is contained in at least one of the blue paths containing $a$, but it must be contained in exactly one since otherwise $a,a'$ would be contained in the exact same blue paths. Since any edges contained in a single blue path cannot be contained in the same blue path, $a'$ and the three inner edges contained in a single blue path must be contained in four separate blue paths. These are the four blue paths that contain edges of the inner $S_7$, so any blue path containing the monochromatic blue edge also contains some edge from a different leg of the inner $S_{13,6}$ which is contained in a single blue path.

For case (iv), there is a red monochromatic edge in the inner $S_7$ so every terminal edge is contained in at least one blue path. Let $a,b$ be the inner edges contained in just one blue path and let $a',b'$ be their incident terminal edges, respectively. Each of $a',b'$ must be contained in a blue path of length $1$ (when restricted to $S_{13,6}$) as otherwise $a,a'$ or $b,b'$ are contained in the exact same blue paths. Any other blue path containing $a'$ ($b'$) also contains $a$ ($b$) so $a'$ ($b'$) is contained in at most two blue paths. Therefore, it suffices to assume $a', b'$ are not monochromatic so each much be contained in at least one red path. In fact, each must be contained in exactly one red path as otherwise, $a,a'$ or $b,b'$ are contained in exactly the same red paths. Thus, $a',b'$, and three of the inner edges are each contained in a single red path. No two edges contained in a single red path can be contained in the same red path, so these are five different paths. The path containing $a' (b')$ but not $a (b)$ is blue, so each of these red paths contains an edge of the inner $S_7$. As there are only five red paths that contain any edges of the inner $S_7$, any path containing the monochromatic red edge also contains some edge from a different leg of the inner $S_{13,6}$ which is contained in a single red path.

Now we consider what happens if some monochromatic edge $e_1$ of the inner $S_7$ is contained only in three paths which each contain some other edge of the inner $S_{13,6}$ appearing in just one path of that color. Without loss of generality, assume that color is red. Let $e_2, e_3$ be the next two edges along the leg of $S_{n,6}$ which contains $e_1$. Edge $e_2$ is contained in a single blue path which does not contain $e_1$. No other blue path can contain $e_2$ and if $e_2$ is monochromatic but only contained in one path, we are done by a previous argument. Thus, $e_2$ must be contained in at least one red path and any such red path also contains $e_1$. If $e_2$ is contained in three red paths, then $e_1,e_2$ are contained in the same red paths and cannot be separated. If $e_2$ is contained in just one red path, then there is no way to separate $e_2$ from the other edge in that path, in a different leg, that is contained in just one red path. Thus, $e_2$ is contained in exactly two red paths. Since there is already a monochromatic red edge, $e_3$ must be contained in a blue path. If the only blue path it is contained in is the one containing $e_2$, then $e_2,e_3$ are contained in the same set of blue paths and cannot be separated. Thus, there is a blue path containing $e_3$ and not $e_2$. There is only one path containing $e_3$ but not $e_2$, since the lower bound of $n+2$ can only be obtained if every edge added in the process of extending $S_{13,6}$ to $S_{n,6}$ increases $c_2$ by exactly $1$. Thus, there is no red path containing $e_3$ but not $e_2$. Edge $e_3$ is contained in at least one red path as otherwise it is a monochromatic edge covered by fewer than three paths. If $e_3$ is contained in both the red paths containing $e_2$, then there is no way to separate $e_2, e_3$. However, if $e_3$ is contained in exactly one of the red paths containing $e_2$, then there is some other edge in the inner $S_{13,6}$ contained in just that red path which cannot be separated from $e_3$. Thus there is no valid way to handle $e_3$, so there exists some spider (and hence infinitely many) with six legs, all of length at least $2$ such that the lower bound is not tight.
\end{document}